\documentclass[11pt]{amsart}

\usepackage[latin1]{inputenc}
\usepackage{amssymb,amsmath,amsfonts}
\usepackage{enumerate}

\usepackage[a4paper,twoside,left=3cm,right=2.8cm,top=3.1cm,bottom=2.3cm]{geometry}

\newtheorem{proposition}{Proposition}[section]
\newtheorem{definition}[proposition]{Definition}
\newtheorem{lemma}[proposition]{Lemma}
\newtheorem{theorem}[proposition]{Theorem}
\newtheorem{corollary}[proposition]{Corollary}
\newtheorem{remark}[proposition]{Remark}

\newtheorem{examples}[proposition]{Examples}

\newcounter{theor}

\newtheorem{teor}[theor]{Theorem}

\DeclareMathOperator{\vol}{Vol}

\DeclareMathOperator{\re}{Re}
\DeclareMathOperator{\im}{Im}

\DeclareMathOperator{\spa}{span}
\DeclareMathOperator{\SL}{SL}
\DeclareMathOperator{\GL}{GL}
\DeclareMathOperator{\U}{U}
\DeclareMathOperator{\SO}{SO}
\DeclareMathOperator{\D}{D}
\DeclareMathOperator{\image}{Im}
\DeclareMathOperator{\supp}{supp}
\DeclareMathOperator{\aff}{aff}

\newcommand{\R}{\mathbb{R}}
\newcommand{\C}{\mathbb{C}}
\newcommand{\N}{\mathbb{N}}
\newcommand{\Z}{\mathbb{Z}}

\newcommand{\K}{\mathcal{K}}
\renewcommand{\l}{\mathrm{length}}
\renewcommand{\d}{\mathrm{diam}}
\newcommand{\hH}{\mathcal{H}}
\newcommand{\cC}{\mathcal{C}}

\DeclareMathOperator{\diam}{\mathrm{diam}}
\renewcommand{\d}{\mathrm{diam}}

\newcommand{\Sa}{\mathcal{S}}

\DeclareMathOperator{\W}{\mathrm{W}}
\DeclareMathOperator{\V}{\mathrm{V}}

\def\lin{\mathop\mathrm{lin}\nolimits}
\def\aff{\mathop\mathrm{aff}\nolimits}
\def\conv{\mathop\mathrm{conv}\nolimits}

\def\cl{\mathop\mathrm{cl}\nolimits}

\def\relint{\mathop\mathrm{relint}\nolimits}

\def\cir{\mathrm{R}}
\def\inr{\mathrm{r}}

\newcommand{\dck}{\D_{C}K}

\newcommand{\func}[5]{\ensuremath{\begin{array}{cccl}
#1:&#2&\longrightarrow&#3\\&#4&\mapsto&#5\end{array}}}

\title[How do difference bodies look like? A geometrical approach.]
{How do difference bodies in complex vector spaces look like? A geometrical approach.}
\author{Judit Abardia} 
\address{Institut f\"ur Mathematik, Goethe-Universit\"at Frankfurt am Main, 
Robert-Mayer-Str. 10, 60054 Frankfurt, Germany}
\email{abardia@math.uni-frankfurt.de}

\author{Eugenia Saor\'in G\'omez} 
\address{Institut für Algebra und Geometrie, Universit\"at Magdeburg, 
Universitätsplatz 2, 39106 Magdeburg, Germany}
\email{eugenia.saorin@ovgu.de}

\begin{document}

\thanks{First author is partially supported by DFG grants AB 584/1-1 and BE 2484/3-1.
Second author is supported by Direcci\'on General 
de Investigaci\'on MTM2011-25377 MCIT and FEDER}

\date{\today}

\subjclass[2000]{Primary 52A20 
52B45; 
Secondary, 52A39 
52A40 
}

\keywords{convex body, complex difference body, Minkowski valuation, dimension, volume, polytope, spherical harmonic}

\begin{abstract}
We investigate geometrical properties and inequalities satisfied by the complex difference body, in the sense of studying which of the 
classical ones for the difference body have an analog in the complex framework.
Among others we give an equivalent expression for the support function of the complex difference body and prove that, unlike the classical case, 
the dimension of the complex difference body depends on the position of the body with respect to the complex structure of the vector space.
We use spherical harmonics to characterize the bodies for which the complex difference body is a ball, we prove that it is a polytope if and only if the two bodies involved in the construction are polytopes and provide 
several inequalities for classical magnitudes of the complex difference body, as volume, quermassintegrals 
and diameter, in terms of 
the corresponding ones for the involved bodies.
\end{abstract}

\maketitle

\section{Introduction}

Let $V$ denote a real vector space of dimension $n$ and $\mathcal{K}(V)$ the space of compact convex sets, shortly, convex bodies, in $V$, endowed with the Hausdorff metric.
The $n$-dimensional
volume of a set $M\subsetneq \R^n$, i.e., its $n$-dimensional Lebesgue
measure, is denoted by $\vol(M)$ (or $\vol_n(M)$ if the distinction of the
dimension is useful). 
For $K,L$ convex bodies, the vectorial or Minkowski sum $K+L$ is also a convex body. It is well known that the Minkowski sum in $\K(V)$ satisfies the cancellation law and behaves nicely with
dilations by positive scalars. 
The \emph{support function of $K\in\K(V)$ in the direction $\xi\in V^{*}$} is defined as $h(K,\xi)=\max\{\xi(x)\,:\,x\in K\}$. 
When $V$ is identified with $\R^n$ and endowed with the euclidean inner product $\langle\cdot,\cdot\rangle$, we denote by $B_n$, the euclidean unit ball and by $S^{n-1}$ its topological boundary.

Given a convex body $K$, its image under reflection in the origin is denoted by $-K$ and the set \[DK=K+(-K)=:K-K\] is a centrally symmetric convex body called
the {\it difference body of $K$}. 

The difference body, seen as an operator $D\!:\K(V)\to \K(V)$ enjoys many important properties. For instance, it is a \emph{Minkowski valuation}, namely, it satisfies $D(K\cup L)+D(K\cap L)=D(K)+D(L)$ whenever $K,L,K\cup L\in\K(V)$, and it is \emph{Minkowski additive}, since $D(K+L)=D(K)+D(L)$. Furthermore, it is \emph{continuous} (with respect to the Hausdorff topology), \emph{translation invariant} (i.e., $D(K+x)=D(K)$ for every $x\in V$) and \emph{$\SL(V,\R)$-covariant} (i.e., $Z(gK)=gZ(K)$, for every $g\in\SL(V,\R)$).

The so-called complex difference body has naturally appeared in the framework of the Minkowski valuation theory. 
Let $W$ be a complex vector space of complex dimension $m\geq 1$, $\alpha\in\C$ and $K\in\K(W)$. 
We denote by $\alpha K$ the convex body obtained under the action of the $m\times m$ complex diagonal matrix $\begin{pmatrix} \alpha &  \dots  &0 \\   & \ddots &  \\ 0 & \dots & \alpha \end{pmatrix}$ on $K$. For a convex body $C\subset\C$, we denote by $d\Sa(C,\cdot)$ the area measure of the convex body $C$ (see, for instance, \cite[Section 4.3]{schneider_book93} for a complete description of the area measure of a convex body).

\begin{definition}Let $W$ be a complex vector space of complex dimension $m \geq 1$ and $C$ a convex body in $\C$. The \emph{complex difference body operator $\D_{C}:\K(W)\to\K(W)$ with respect to $C$} is defined by 
\begin{equation} \label{eq_def_pi_c}
 h(\D_{C}\!K,\xi)=\int_{S^1}h(\alpha K,\xi)d\Sa(C,\alpha),\quad \xi \in W^*.
\end{equation}
\end{definition}

By simplicity, we will usually refer to complex difference body operator with respect to $C$ only by {\it complex difference body operator}. We notice that Equation \eqref{eq_def_pi_c} can also be written as 
\begin{equation}\label{e:DcK vectorial integration}
D_{C}K=\int_{S^1}\alpha K d\Sa(C,\alpha).
\end{equation}
We will also consider the operator $\D:\K(\C)\times\K(W)\to\K(W)$, defined by $\D(C,K)=\dck$, which will be again called complex difference body operator.

The following result characterizes the complex difference body operator for dimension $m\geq 3$. 

\begin{teor}[\cite{abardia12}]\label{thm_covariant}
Let $W$ be a complex vector space of complex dimension $m \geq 3$. A map $Z:\mathcal{K}(W) \to \mathcal{K}(W)$ is a continuous translation invariant and $\SL(W,\mathbb{C})$-covariant Minkowski valuation if and only if there exists a unique planar convex body $C$ with Steiner point $s(C)=0$ such that $Z=\D_C$.
\end{teor}

The motivation to study the previous classification result was given by Ludwig's characterization theorem for the difference body (and projection body) operator. More precisely, from \cite{ludwig05}, it follows that in a real vector space of dimension $n\geq 2$ the difference body is, up to positive constants, the unique continuous Minkowski valuation which is translation invariant and $\mathrm{SL}(V,\mathbb{R})$-covariant (see also \cite{ludwig02}).
In \cite{abardia13, abardia12,abardia.bernig}, the complex analog of Ludwig's result was studied, obtaining, in particular, Theorem \ref{thm_covariant}.

\smallskip
We remark that the ``usual'' difference body of a convex body $K$, i.e., $DK$, is indeed obtained as a particular case of $\D_{C}K$, namely for $C=[-i/2,i/2]$. (Here, and along the paper, we identify $\R^2\equiv \C$.) In fact, the unit normal vectors to $C$ are $(1,0)$ and $(-1,0)$, which correspond to $1$ and $-1$ in $\C$. Thus,   
the surface area measure of $C$ is $\Sa(C,\cdot)=\delta_{1}+\delta_{-1}$ and the statement follows.

In view of the previous remark and the natural definition of the complex difference body operator in a complex setting, we aim to shed light on the geometry of this newly defined operator. To achieve this objective, we study geometrical properties and inequalities satisfied by the complex difference body, comparing them with the well-known properties and inequalities of the real difference body, which we recall in the next lines.

Let $K\in\K(V)$ and $V$ be endowed with the euclidean inner product. 
It is easy to see that for every $u\in S^{n-1}$, $h(DK,u)=w(K,u)$, where $w(K,u):=h(K,u)+h(K,-u)$ denotes the \emph{width of $K$ in the direction $u$}.
As a consequence, $DK$ is a ball if and only if $K$ has constant width and $DK$ is homothetic to $K$ if and only if $K$ is centrally symmetric. We will denote centrally symmetric convex bodies by symmetric convex bodies.

The classical Brunn-Minkowski inequality 
yields that
\begin{equation}\label{e:lhs RS}
2^n \vol_n(K)\leq \vol_n(DK)
\end{equation}
with equality precisely if $K$ is centrally symmetric.
An upper bound for $\vol_n(DK)$ is given by the so called {\it Rogers-Shephard inequality}, which establishes that
\begin{equation}\label{e:rhs RS}
\vol_n(DK)\leq \binom{2n}{n}\vol_n(K)
\end{equation}
with equality precisely if $K$ is a simplex (see \cite{rogers.shephard} and \cite[Section 7.3]{schneider_book93}).

To study similar facts for the complex difference body operator, we have to remark, first of all, that the dimension of the complex difference body of $K$ depends on the position of $K\in\K(W)$ as a subset of $W$, that is, the dimension of $\dck$ is not rotation-invariant, except for the case in which $C$ is a segment. The following result 
gives the precise dimension of $\D_CK$.
\begin{theorem}\label{dimDC}Let $K\in\K(W)$ be a convex body of dimension $l$ contained in an $l$-dimensional subspace $E$. Let $a:=\dim_{\C}(E\cap JE)$. Then, for any $C\in\K(\C)$ with interior points, it holds
$$\dim\D_{C}K=2(l-a).$$ 
Here $J$ denotes a complex structure of the complex vector space $W$.
\end{theorem}
This implies that a Rogers-Shephard's type inequality cannot hold for the complex difference body. However, we provide sharp (lower) bounds for 
the volume and other quermassintegrals of $\D_{C}K$ in terms of the quermassintegrals of both, $K$ and $C$ in Section \ref{s:geom ineq}.

The classical difference body of $K$ is a symmetric convex body. In the complex case, this is no longer true, in general. Indeed, the image of $\D_{C}$ is contained in the space of symmetric convex bodies if and only if $C$ itself is a symmetric one (see Corollary \ref{cor:osim}). In turn,
in Lemma \ref{l: width} we will prove that for any $C\in\K(\C)$ with $\l(C)=1$, 
the width of $\D_{C}K$ does not decrease whereas the diameter does not increase,
 which allows us to say, roughly, that the complex difference body of $K$ with respect to $C$ is closer to a
ball than $K$, for any $C$.

As mentioned above, $DK$ is a ball if and only if $K$ is a convex body of constant width. It is also well known that $DK$ is a polytope if and only if $K$ is also a polytope.
We address these questions and similar ones for the complex difference body proving, for instance, the following result.
\begin{theorem}\label{DcK polytope} 
Let $C\in\K(\C)$ and $K\in\K(W)$. Then, $\D_{C}K$ is a polytope if and only if both $C$ and $K$ are polytopes.
\end{theorem}
As a consequence we obtain the following
\begin{corollary}
The map $\D\!: \K(\C)\times\K(W)\to \K(W)$ is neither injective nor surjective.
\end{corollary}
This is the analog to the well known fact that the classical difference body operator is neither injective nor surjective on $\K(V)$. We further study the range of injectivity of $\D_{C}$ for a fixed $C$.

We characterize the pairs $(C,K)$ for which $\D_{C}K$ is a ball. In order to do so, we calculate the harmonic expansion of the support function of $\D_CK$,
which is given in terms of the Fourier coefficients of the harmonic expansion of the support functions of $C$, $c_{j}(h_{C})$, and $K$, $\pi_{j,l}(h_{K})$, for $j,l\geq 0$ (see Sections \ref{s: sph harm} and  \ref{s:multiplier} for precise definitions and details).
\begin{corollary}Let $C\in\K(\C)$ and $K\in\K(W)$. Then, $\D_{C}K$ is a ball if and only if  for every $(j,l)\in\N^2\backslash \{(0,0)\}$, either $c_{j-l}(h_{C})=0$ or $\pi_{j,l}(h_{K})=0$.
\end{corollary}
The previous corollary implies that if $C$ is a symmetric convex body with $c_{2j}\neq 0$ for every $j\in\Z$, 
then $\dck$ is a ball if and only if $K$ is of constant width, 
recovering the classical case. 
The same ideas allow us to characterize when $\D_CK $ is of constant width, symmetric or $S^1$-invariant (see Corollaries \ref{cor:cw}, \ref{cor:rtinv} and \ref{cor:osim}).

A convex body $K\subset W$ is said to be \emph{$S^1$-invariant} if $\alpha K=K$ for every $\alpha\in S^1$. The $S^1$-invariant convex bodies are also called $R_{\theta}$-invariant or equilibrated bodies. This class of bodies is well-adapted to the complex structure of the ambient space (see, for instance, \cite{koldobsky_koenig_zymonopoulou08,koldobsky.paouris.zymonopoulu,rubin10}) and shall be used along the paper. Related results concerning convex bodies or valuations in a complex vector space can be found in \cite{alesker03,bernig_fu_hig,fu06,koldobsky_koenig_zymonopoulou08,koldobsky.paouris.zymonopoulu,rubin10,wannerer13}.

The paper is organized as follows. In Section \ref{s: func prop} we give some equivalent descriptions of the complex difference body, and prove some auxiliary results as well as Theorems \ref{dimDC} and \ref{DcK polytope}. Section
\ref{s:geom ineq} is devoted to prove several inequalities involving the complex difference body and magnitudes such as quermassintegrals, circumradius and diameter. In Section \ref{s: sph harm} we introduce the necessary notation
and results from the theory of spherical harmonics, which will be used later on. In Section \ref{s:multiplier} we prove that the extension of the complex difference body operator to the space of continuous functions on the sphere is a multiplier transformation (in the context of spherical harmonics). As a consequence of this, we provide the coefficients of the harmonic expansion of the support function of the complex difference body. Section \ref{s: fixed points} concerns the study of the fixed points of the complex difference body operator and the relation of the latter to some particular classes of convex bodies, such as universal convex bodies or $M$-classes. In the last section we deduce Corollaries 1.3 and 1.4, as well as some other related properties of $\dck$.

\section{Functional properties of $\D_{C}$}\label{s: func prop}

In this section we study some functional properties of the complex difference body. 
From now on, $V$ denotes always a real vector space of dimension $n$ and $W$ a complex vector space of complex dimension $m$. We will often identify $W$ with a real vector space of dimension $2m$. If a basis of $W$ is fixed, then we use the following map for the identification
\begin{equation}\label{idWV}
(w_{1},\dots,w_{m})=(w_{11}+iw_{12},\dots,w_{m1}+iw_{m2})\mapsto(w_{11},w_{12},\dots,w_{m1},w_{m2}).
\end{equation}
By $V^*$ and $W^*$ we denote the dual vector spaces of $V$ and $W$, respectively.

First we deal with equivalent constructions or descriptions of the complex difference body. The next proposition provides an equivalent expression for the operator $\D_{C}$ (cf. \cite{abardia.bernig}). 

\begin{proposition}
\label{prop_n2}
Let $\mu:\K(\C)\to\R$ be a continuous, translation invariant, monotone valuation homogeneous of degree $1$. Then the operator $Z:\K(W) \to \K(W)$ defined by
\begin{displaymath}
 h(ZK,\xi)=\mu(\xi(K)),\quad K \in \K(W),\, \xi\in W^*
\end{displaymath}
is a continuous, translation invariant, $\SL(W,\C)$-covariant Minkowski valuation. $\xi(K)$ is defined as $\xi(K):=\{z\in\C\,|\,z=\xi(k),k \in K\}$ and is a convex body in $\mathbb{C}$. 
\end{proposition}

\begin{proof}
As $K$ is a convex body in $W$ and $\xi$ is linear, we have that $\xi(K)$ is also a convex body. Moreover, the function $\xi\mapsto \mu(\xi(K))$ is 1-homogeneous in $\xi$. It is also subadditive, and hence the support function of some convex body $ZK\subset W$. 
 Indeed, we have $(\xi_{1}+\xi_{2})(K)\subset \xi_{1}(K)+\xi_{2}(K),$ hence the monotonicity of $\mu$ gives 
$$\mu((\xi_{1}+\xi_{2})(K))\leq \mu(\xi_{1}(K)+\xi_{2}(K))=\mu(\xi_{1}(K))+\mu(\xi_{2}(K)),$$
where the last equality holds since $\mu$ has degree of homogeneity 1, and, thus it is Minkowski additive (cf. \cite[Theorem 3.2]{goodey_weil84}).

To prove the valuation property, let $K,L, K \cup L \in \K(W)$. Then, since $\mu$ is a valuation,  
\begin{align*}
\mu(\xi(K \cup L))+&\mu(\xi(K \cap L))  =\mu(\xi(K) \cup \xi(L)) + \mu(\xi(K) \cap \xi(L)) = \mu(\xi(K))+\mu(\xi(L)).
\end{align*}
The translation invariance and the continuity of $Z$ follow from the linearity of $\xi$ and the corresponding properties of $\mu$.
Moreover, $Z$ is $\SL(W,\C)$-covariant, since for every $g \in \SL(W,\C)$ and $\xi \in W^*$ it holds, from the properties of the support function (see \cite{schneider_book93}),
$$h(ZgK,w)=\mu(\xi(gK))=\mu((g^*\xi)(K))=h(ZK,g^*\xi)=h(gZK,\xi).$$
\end{proof}

\begin{remark}\label{mu} The examples provided by the previous proposition are (monotone) Minkowski valuations satisfying the properties of Theorem \ref{thm_covariant}. Thus, given $\mu$ there exists a convex body $C$ with $h(\D_{C}\!K,\xi)=\mu(\xi(K))$.
\end{remark}

The monotone valuation $\mu$ appearing above can be given more explicitly using mixed volumes of planar convex bodies. This gives a more geometric expression for the support function of the complex difference body, as the following proposition shows.

\begin{proposition}\label{p: h eq vol mixto}For $K\in\K(W)$ and $C\in\K(\C)$, the following equality holds
\begin{equation}\label{V2}
h(\D_{C}K,\xi)=V_{2}(\xi(K),\overline C).
\end{equation}

If $W$ is endowed with a hermitian scalar product, then 
$$h(\D_{C}K,\xi)=V_{2}(K|\xi_{\C},\overline C),$$
where $K|\xi_{\C}$ denotes the projection of $K$ onto the 2-dimensional real space in $W$ generated by $\{\xi,-J\xi\}$, identified with $\C$.

If $m=1$, then for $\xi\in S^1$, 
\begin{equation}\label{V2Rotation}
h(\D_{C}K,\xi)=V_{2}(R_{\overline{\xi}}K,\overline C),
\end{equation}
where $R_{\xi}$ denotes the rotation of $K$ defined by $\xi$.
\end{proposition}
\begin{proof}
By definition, we have
\begin{align*}h(\D_{C}\!K,\xi)&=\int_{S^1}h(\alpha K,\xi)d\Sa(C,\alpha)=\int_{S^1}h(K,\overline{\alpha}\xi)d\Sa(C,\alpha)\\&=\int_{S^1}h(\xi(K),\alpha)d\Sa(\overline C,\alpha)=V_{2}(\xi(K),\overline C),\end{align*}
where we have used
\begin{align*}h(K,\overline{\alpha}\xi)&=\sup_{k\in K}\{\re (\overline{\alpha}\xi)(k)\}=\sup_{k\in K}\{\re(\overline{\alpha}(\xi(k)))\}=\sup_{l\in \xi(K)}\{\re(\overline{\alpha}(l))\}\\&=h(\xi(K),\overline{\alpha}),\end{align*}
with $\xi(K)=\{z\in\C\,|\,z=\xi(k),k \in K\}$, as in Proposition \ref{prop_n2}.

Assume that $W$ is endowed with a hermitian scalar product
$(\cdot,\cdot)$. With the identification of $\C^m$ with $\R^{2m}$ given in \eqref{idWV}, there exists a scalar product $\langle\cdot,\cdot\rangle$ in $\R^{2m}$ for which 
$$(x,y)=\langle x,y\rangle +i\langle -Jx,y\rangle.$$
Therefore,
$$\xi(k)=(\xi,k)=\langle\xi,k\rangle+i\langle -J\xi,k\rangle\in\C,$$
and $\xi(k)\xi\in\C^m$ is the orthogonal projection of $k$ onto the
subspace generated by $\xi$ and $-J\xi$.

For the case $m=1$ it remains to prove that $\xi(K)=R_{\overline{\xi}}(K)$. This follows from
$$\xi(K)=\{\xi(k)\,:\,k\in K\subset\K(\C)\}=\{(\xi_{1}k_{1}-\xi_{2}k_{2})+i(\xi_{1}k_{2}+\xi_{2}k_{1})\}=R_{\overline{\xi}}K.$$
\end{proof}

Using the above description of the support function of $\D_{C}K$ we obtain the following result.

\begin{corollary}\label{c:monotonicity}
The operator $\D\!:\K(\C)\times \K(W)\to \K(W)$ is monotone in its two variables.
\end{corollary}

Now we prove Theorem \ref{dimDC}.
\begin{proof}[Proof of Theorem \ref{dimDC}]
Let  $g\in\SL(W,\C)$, $K$ an $l$-dimensional convex body in $W$ and  
$E=\aff(K)$ the real affine hull of $K$. As $\dim\D_{C}gK=\dim
g\D_{C}K=\dim\D_{C}K$, and $\SL(W,\C)$ acts transitively on the
subspaces of dimension $l$ having $\dim(E\cap JE)$ fixed, we can suppose that
\[E=\spa_{\R}\{e_{1},Je_{1},\dots,e_{a},Je_{a},e_{a+1},e_{a+2},\dots,e_{l-a}\}\]
for some vectors $e_{1},\dots,e_{l-a}$ linearly independent over
$\C$.

We claim that for any direction $e_{i}\in E$, the direction $Je_{i}$ belongs to $\aff(\dck)$, the minimal subspace in $W$ containing $\dck$. Indeed, without loss of generality, we can assume that $0\in\relint(K)$, where $\relint(K)$ denotes the relative interior of $K$. 
Thus, there exists $\epsilon>0$ such that $L_{i}:=[-\epsilon e_{i},\epsilon e_{i}]\subset K$ for every $i=1,\dots, l-a$. Using the monotonicity of $\D_{C}$, we have that $\D_{C}L_{i}\subset\dck$. Hence, it suffices to prove that there exists $\lambda>0$ for which $\lambda Je_{i}\in \D_{C}L_{i}$ to have that $Je_{i}\in\aff(\dck)$. From the monotonicity of the support function, this is equivalent to show that $0<\lambda\leq h(\D_{C}L_{i},Je_{i})$ (since $0$ is in the relative interior of $\D_{C}L_i$). 

The nonnegativity of the area measure of $C$ yields
$$h(\D_{C}L_{i},Je_{i})=\epsilon\int_{S^1}|\langle \alpha e_{i},Je_{i}\rangle|d\Sa(C,\alpha)=\epsilon\int_{S^{1}}|\im(\alpha)|d\Sa(C,\alpha)\geq 0,$$
with equality if and only if $\Sa(C,\cdot)$ is concentrated on $1,-1$, that is, if and only if, $C$ is an interval on the imaginary axis of $\C\equiv\R^2$.

Thus, the claim follows for every convex body $C\in\K(\C)$ with interior points. An analogous argument shows that for every $e_{i}\in E=\aff(K)$, it also holds $e_{i}\in\aff(\dck)$, $C$ with interior points. Therefore, we get 
$$\dim\D_{C}K\geq 2a+2(l-2a)=2(l-a),$$
with $2a$ the number of directions
$\{e_{1},Je_{1},\dots,e_{a},Je_{a}\}$ and $l-2a$ the number of
directions $\{e_{a+1},\dots,e_{l-a}\}$.

Let $\tilde E:=\{e_{1},Je_{1},\dots,e_{l-a},Je_{l-a}\}$ and $F:=\{e_{l-a+1},Je_{l-a+1},\dots,e_{m},Je_{m}\}$ be so that $\{e_{1},\dots,e_{l-a},e_{l-a+1},\dots,e_{m}\}$ constitutes an orthonormal basis of the complex vector space $W$. It remains to prove that $\dck\subset \tilde E$. Since $W=\tilde E\oplus F$, it suffices to show that $h(\dck,v)=0$ for every $v\in F$ 
(see \cite[Section 1.7]{schneider_book93}). 

Let $v\in F$. Using that $K\subset E$ and $\alpha K\subset \alpha E\subset\tilde E$, we have 
$$h(\dck,v)=\int_{S^1}h(\alpha K,v)d\Sa(C,\alpha)=0,$$
and $\dim\dck=2(l-a)$.
\end{proof}

For $\rho\in\C$, using the definition of {\it scaling} a convex body by
a complex number mentioned in the introduction, we can ask for the behavior
of the complex difference body operator when either, $C\in\K(\C)$ or $K\in\K(W)$ are
{\it scaled} by $\rho$. This is the content of the next result, which shall be used without further mention along of the paper.

\begin{lemma}\label{p: rho C}
\hfill
\begin{enumerate}
\item Let $K\in\K(W)$, $C\in\K(\C)$ and $\rho\in\C$ (not
necessarily in $S^1$). Then,
\[
\D_{\rho C}K= \D_{C}(\rho K)=\rho \D_{C}K.
\]

\item If $\rho\in S^1$ and $C=B_2$ is the unit ball in $\C$, then
\[
\rho \D_{B_2}K=\D_{B_2}K,
\]
that is, $\D_{B_{2}}K$ is $S^1$-invariant for every
$K\in\K(W)$. 

\item If $\rho_{0} C=C$ for some $\rho_{0}\in\C$,
then, for every $K\in\K(W)$,
$$\rho_{0} \D_{C}K=\D_{C}K.$$

\item Let $\varphi=T+p$ with $T\in\GL(W,\C)$ and $p\in W$. Then,
\[
\D_{C}(\varphi K)=\D_{C}(TK)=T\D_{C}K.
\]
\item If $K=B_{2m}$, where $m=\dim_{\C} W$, then $D_C K=\l(C)B_{2m}$.
\end{enumerate}
\end{lemma}

\begin{proof}
To prove i), let $\rho=|\rho|\tilde\rho$, with $\tilde\rho\in S^1$. Then, we have
\begin{align*}
h(\D_{\rho C}K,\xi)=\int_{S^{1}}h(\alpha K,\xi)d\Sa(\rho C,\alpha)
=\int_{S^1}|\rho|h(\tilde{\rho}\beta
K,\xi)d\Sa(C,\beta)=h(\D_{C}(\rho K),\xi).
\end{align*}
Statements ii) and iii) follow directly from i); iv) from the translation invariance and the $\SL(W,\C)$-covariance of $\D_{C}$; and v)  from the rotation invariance of $B_{2m}$ and the fact that $\int_{S^1}{d\Sa(C,\alpha)}=\l(C)$.
\end{proof}

Next, we observe the following fact which easily arises from the previous results.

\begin{corollary}
$\D\!:\K(\C)\times\K(W)\longrightarrow \K(W)$ is a
continuous, translation invariant, $\SL(W,\C)$-covariant Minkowski
valuation in each component.
\end{corollary}

In the following, we would like to focus on a special type of convex bodies, namely polytopes. We shall need the following lemma which holds in $\C$ for any convex body.

\begin{lemma}\label{l: conv measures}Let $C,K\in\K(\C)$. Then,
$$\Sa(\D_{C}K,\cdot)=\Sa(C,\cdot)*\Sa(K,\cdot),$$
where $*$ denotes the convolution between the (nonnegative)
surface area measures $\Sa(C,\cdot)$ and $\Sa(K,\cdot)$.
\end{lemma}
\begin{proof}In the sense of distributions, if necessary, we can write (see, for instance \cite{berg} or \cite[p.\! 110]{schneider_book93})
$$\Sa(\dck,\cdot)=h_{\dck}''+h_{\dck},$$
where we are interpreting $h_{\dck}$ as a $2\pi$-periodic function on $\R$. 
Let $\varphi\in\mathcal{C}^{\infty}(\R)$ be a $2\pi$-periodic test
function. Using the definition of the derivative of a
distribution, the definition of $\D_{C}K$ and the Fubini theorem,
we get
\begin{align*}(h_{\D_{C}K}''+h_{\D_{C}K})(\varphi)&=\int_{0}^{2\pi}(h_{\D_{C}K}''+h_{\D_{C}K})(t)\varphi(t)dt=
\int_{0}^{2\pi}h_{\D_{C}K}(t)(\varphi''+\varphi)(t)dt
\\&=\int_{0}^{2\pi}\int_{0}^{2\pi}h_{K}(t-s)(\varphi''+\varphi)(t)dt d\Sa(C,s)
\\&=\int_{0}^{2\pi}\int_{0}^{2\pi}\varphi(\tilde t-s)(h_{K}''+h_{K})(\tilde t)d\tilde t d\Sa(C,s)
\\&=\int_{0}^{2\pi}\int_{0}^{2\pi}\varphi(\tilde t-s)d\Sa(K,\tilde t)d\Sa(C,s)
\\&=(\Sa(C,\cdot)*\Sa(K,\cdot))(\varphi).
\end{align*}
\end{proof}

Next we prove the stated property that $\D_CK$ is a polytope if and only if both $K$ and $C$ are polytopes.

\begin{proof}[Proof of Theorem \ref{DcK polytope}]
The only if part is easy. Indeed, let $C\in\K(\C)$ be the polytope with outer normal vector $\alpha_{i}$ to the facet (edge) $f_{i}$ having length $s_i$, $i=1,\dots,N$. Then
\begin{equation}\label{eqCpol}
h(\D_{C}K,\cdot)=\int_{S^1}h(\alpha
K,\cdot)d\Sa(C,\alpha)=\sum_{i=1}^Ns_{i}h(\alpha_{i} K,\cdot)=h(\sum_{i=1}^N s_{i}\alpha_{i} K,\cdot).
\end{equation}
Since $K$ is a polytope, $\alpha_{i}K$ is a polytope for every $i=1,\dots, N$ and
$\sum_{i=1}^N s_{i}\alpha_{i} K$ is a polytope too.

For the if part, assume first that $m=1$. It is known (see
\cite[Theorem 4.6.4]{schneider_book93}) that a convex body $L\in\K(\C)$ with surface area measure
$\Sa(L,\cdot)$ is a polygon if and only if
\[
\Sa(L,\cdot)=\sum_{i=1}^{N}r_{i}\delta_{u_{i}},
\] 
where $r_{i}>0$, $u_{i}\in S^1$ such that
$\sum_{i=1}^{N}r_{i}u_{i}=0$ and the linear subspace generated by $\{u_{i}\}_{i=1}^{N}$ satisfies $\lin\{u_i, i=1,\dots,N\}=\R^2$. Thus, using Lemma \ref{l: conv measures}, $\D_{C}K$ is a polygon if and
only if $\Sa(K,\cdot)*\Sa(C,\cdot)$ is a finite positive linear
combination of Dirac measures, that is, the measure
$\Sa(\D_{C}K,\cdot)$ has discrete support. 
It is known (see \cite[p.\! 194]{garnir}) that for positive Borel measures
$S,T$, we have
\[
\supp(S*T)=\cl\{\supp(S)+\supp(T)\},
\] 
where $\cl\{A\}$ denotes the closure of the set $A$.
Then, we have that $\Sa(\D_{C}K,\cdot)$ is a finite linear combination of
Dirac measures if and only if $\Sa(C,\cdot)$ and
$\Sa(K,\cdot)$ are also of this form, which concludes the proof for $m=1$.

For general $m\geq 2$, let $\xi\in S^{2m-1}$. By assumption, the
2-dimensional convex body $(\D_{C}K)|\xi_{\C}$ is a polygon and
$$h((\D_{C}K)|\xi_{\C},\alpha\xi)=h(D_{C}K,\alpha\xi)=
V_{2}(K|(\alpha\xi)_{\C},\overline{C})=V_{2}(R_{\overline\alpha}(K|\xi_{\C}),\overline{C})=h(\D_{C}(K|\xi_{\C}),\alpha).$$

Thus, using the case $m=1$,
we can assure that $\overline{C}$ is a polytope. Hence, $C$ is a polytope too and 
the support function of $\D_{C}K$ is given by \eqref{eqCpol}. Therefore,  
$\D_{C}K$ is a finite linear combination of rotations and
dilations of $K$. Since the only convex bodies which are summands of polytopes, are 
polytopes, $K$ is itself a polytope (see \cite[Chapter 15.1]{Gruenbaum}).
\end{proof}

Restricted to zonotopes, the following statement holds.

\begin{proposition}\label{zonotope2}
\hfill
\begin{enumerate}
\item Let $K$ be a zonotope. Then, $\D_{C}K$ is a zonotope if and
only if $C$ is a polygon.

\item Given $C$ a zonotope, there exists $K\in\K(W)$ such that $\D_CK$ is a zonotope, but
$K$ is not.
\end{enumerate}
\end{proposition}

\begin{proof}
\begin{enumerate}
\item Without loss of generality let $K=\displaystyle\sum\limits_{i=1}^N  r_i[-u_i,u_i]$ with $r_{i}>0$ and $u_{i}\in S^{2m-1}$. It holds,
\[
h(\D_{C}K,\cdot)=\sum_{i=1}^{N} r_{i}\int_{S^1}|\langle \alpha u_{i},\cdot\rangle|d\Sa(C,\alpha),
\]
which is the support function of a zonotope if and only if
$\Sa(C,\cdot)$ is a discrete measure (see for example \cite[Section 3.5]{schneider_book93}), that is, if and only if $C$ is
a polytope.

\item By linearity of $C\mapsto\dck$, it is enough to consider in the plane the case $C=[-i/2,i/2]$, a segment and $K=\conv\left\{\begin{pmatrix}0 \\ 0\end{pmatrix}, \begin{pmatrix} 1 \\ 0\end{pmatrix}, \begin{pmatrix}0 \\ 1\end{pmatrix}\right\}$ a triangle,
which, of course is not a zonotope.
Then, $$D_C K =K-K=\left[\begin{pmatrix}0 \\ 0\end{pmatrix}, \begin{pmatrix}1 \\ 0\end{pmatrix}\right]+\left[\begin{pmatrix}0 \\ 0\end{pmatrix}, \begin{pmatrix} -1 \\ 1\end{pmatrix}\right]+\left[\begin{pmatrix}0 \\ 0\end{pmatrix}, \begin{pmatrix}0 \\ -1\end{pmatrix}\right].$$ 
\end{enumerate}
\end{proof}

Although it is not connected with polytopes, we include next result, which ensures that if $K$ is a convex body of class $\cC_+^2$, then $\dck$ is of class $\cC_+^2$. A convex body $K$ is said to be \emph{of class $\cC_{+}^2$} if its boundary is a submanifold of class $\cC^2$ and all principal curvatures are strictly positive at every point of the boundary. 

\begin{proposition}
\hfill
\begin{enumerate}
\item Let $K\in\K(W)$ be of class $\cC_+^2$.
Then, $\dck$ is of class $\cC_+^2$ for every  $\cC\in\K(\C)$.

\item Let $C\in\K(\C)$ be of class $\cC_{+}^2$. Then, $\dck$ is of class $\cC^2_{+}$ for every $K\in\K(W)$ if and only if $\dim_{\C}W=1$. 

\item Let $C\in \K(\C)$ and $K\in\K(W)$ be so, that $\dck=\l(C)K$ (for instance, if $K$ is $S^1$-invariant, cf. Theorem \ref{fixedPoints}). Then $\dck$ is of class $\cC^2_{+}$ for every $C\in\K(\C)$ if and only if $K$ is of class $\cC^2_{+}$.

\end{enumerate}
\end{proposition}

\begin{proof}
\begin{enumerate}

\item Since $K$ is of class $\cC^2_+$, we have that the Hessian matrix $d^{2}(h_{K})_{u}$ is positive definite for every $u\in S^{2m-1}$ (see \cite[p.~107]{schneider_book93}). 
From the definition of the support function of $\dck$, we have, by linearity, $$d^{2}(h_{\dck})_{u}=\int_{S^1}d^{2}(h_{\alpha K})_{u}d\Sa(C,\alpha),$$
which is again a positive definite matrix, since $d^{2}(h_{\alpha K})_{u}$ is positive definite for every $\alpha\in S^1$ and $\Sa(C,\cdot)$ a non-negative measure. 

\item If $\dim_{C}W=1$, then we have, by Lemma \ref{l: conv measures} and Corollary 7.25 in \cite{schneider_book93} that $\dck=\D_{K}C$  (cf.\! Lemma \ref{l:dc segment}). Using i), we obtain the result. 

For $\dim_{\C}W>1$, it suffices to take $K=B_{2}\times B_{2}\subset\C^2$, the cartesian product of two balls contained in orthogonal subspaces $E_{1}$ and $E_{2}$ with $E_{1},E_{2}\cong\C$. Since $K$ is a $S^1$-invariant convex body, we have $\dck=\l(C)K$ for every $C$. Thus, the statement follows since $K$ is not of class $\cC^{2}_{+}$.

\item It follows directly from the fact $\dck=\l(C)K$ and i).

\end{enumerate}
\end{proof}

In the same line, it follows directly from Lemma \ref{p: rho C}, that if $K=L+\epsilon B_{2m}$, $\epsilon>0$, then $\D_{C}K=\D_{C}L+ \epsilon\l(C)B_{2m}$.

The analogue study for symmetric or $S^1$-invariant convex bodies needs the tools contained in Section \ref{s: sph harm} and will be treated
after that.

\section{Geometric inequalities for $\D_C \,K$}\label{s:geom ineq}
In this section we intend to study some of the most classical
magnitudes of convex bodies for the complex difference body $\dck$
in dependence of the corresponding magnitudes for $K$ and $C$.

\begin{lemma}\label{l:dc segment}
Let $W$ be a complex vector space with $\dim_{\C} W=m$. Let $u\in
S^{2m-1}$ and $C\in\K(\C)$. 
\begin{enumerate}
\item If $m=1$, then $$\dck=\D_{K}C \quad\forall\, K\in\K(\C).$$

\item For $m\geq 1$, if $K$ is a line segment, i.e., $K=[au,bu]$, $a< b$, $u\in S^{2m-1}$, we have
\[
\dck=\frac{b-a}{2}\D_{[-1,1]}(C\cdot u)=(b-a)D(iC\cdot u),
\]
where $C\cdot u:=\{cu\in W\,:\,c\in C\subset\C\}$.
\end{enumerate}
\end{lemma}

\begin{proof}
\begin{enumerate}
\item The statement follows directly from Lemma \ref{l: conv measures} and Corollary 7.2.5 in \cite{schneider_book93}.
We include a more geometric proof of this fact which uses Proposition \ref{p: h eq vol mixto}.
Let $K\in\K(\C)$ be a convex body. Then, using
\eqref{V2Rotation} we can write 
\[
h(D_C K,\xi )=\V_{2} (R_{\overline{\xi}}K,\overline{C})
\] 
and 
\[
h(D_K C,\xi )=\V_{2} (R_{\overline{\xi}}C,\overline{K}).
\]
Using known properties of mixed volumes (see for example \cite[Chapter 5]{schneider_book93}), it follows, on the one hand,
that $h(\dck,\xi )=\V_{2}(R_{\overline{\xi}}K,\overline{C})=\V_{2}(K,R_{\xi}\overline{C})$. On the other hand,
$h(\D_K C,\xi )=\V_{2}(R_{\overline{\xi}}C,\overline{K})=\V_{2}(\overline{R_{\overline{\xi}}C},K)$. Thus, it is enough to
prove that $\overline{R_{\overline{\xi}}C}=R_{\xi}\overline{C}$ which follows directly taking $R_{\xi} C = e^{i\theta} C$
for some $\theta\in [0,2\pi)$.

\item Theorem \ref{dimDC} yields that $\dck\subset \lin_{\C} \{u\}$. Thus, it suffices to compute the support function of $\dck$ in the directions $\beta u$, $\beta\in S^1$ and to compare it with the support function of $\D_{[-1,1]}(C\cdot u)$. Using the translation invariance of $\dck$ and the first statement of this lemma, we get 
\begin{align*}
\quad\quad\quad h(\dck,\beta u)&=\frac{b-a}{2}\int_{S^1}h([-\alpha u,\alpha u],\beta u)d\Sa(C,\alpha)
=\frac{b-a}{2}\int_{S^{1}}h([-\alpha,\alpha],\beta)d\Sa(C,\alpha)
\\&
=\frac{b-a}{2}h(\D_{C}[-1,1],\beta)=\frac{b-a}{2}h(\D_{[-1,1]}C,\beta),
\end{align*}
and 
$$h(\D_{[-1,1]}(C\cdot u),\beta u)=2\left(h(i(C\cdot u),\beta u)+h(-i(C\cdot u),\beta u)\right)=h(\D_{[-1,1]}C,\beta).$$
\end{enumerate}
\end{proof}

For the usual difference body of $K$, basic properties of the
Minkowski sum ensure that, up to a translation, $K$ is contained
in $K-K$. The question thus arises, whether there is a positive
constant $\lambda$, maybe depending on $K$ and $C$, such that
$\lambda K\subset \dck$. 

\begin{proposition}\label{p: lambda K subset dck}
Let $K\in\K(W)$ be a convex body of dimension $l$ contained in an $l$-dimensional real subspace
$E\subset W$, and $C\in \K(\C)$ with interior points. Then, there exists $\lambda(K,C) >0 $ so that, up to translation, $\lambda(K,C) K \subset
\dck$.
\end{proposition}

\begin{proof}
We distinguish two cases. First, if $l=2m$ and $\inr(K)$ denotes the inradius of $K$, then
$\inr(K)B_{2m}\subseteq K$ (up to a translation of $K$) and using Corollary
\ref{c:monotonicity} we have that
\[
\l(C)\inr(K) B_{2m}=\inr(K)\D_C{B_{2m}}\subseteq \dck.
\]
In order to get $\lambda(K,C)$, we use the notion of relative inradius. The \emph{relative inradius $r(K; L)$ of  $K$ with respect to $L$}, $K,L\in\K(\R^n)$, is defined by $$r(K; L)=\sup\{r\,:\,\exists\, x\in\R^n\textrm{ with }x+rL\subset K\}.$$
Thus, it is enough to consider
$\lambda(K,C)=\l(C)\inr(K)\inr(B_{2m}; K)>0$, being $\inr(B_{2m}; K)$ the relative inradius of $B_{2m}$ with respect to $K$. We notice that in this case, $\lambda(K,C)$ depends only on $K$ and the length of $C$ but not its shape. 

For the case $l<2m$, let $E=\aff(K)\subset W$. We denote by
$\inr(K; B_l)$ the relative inradius of $K$ with respect to $B_l$ (in the
ambient space $E$). The monotonicity of $\dck$ yields
\[
\inr(K; B_l)\D_C{B_l}\subseteq \dck.
\]
To prove the existence of $\lambda(K,C)$, arguing as in the
previous case, it will be enough to prove that $\D_C{B_l}$ has interior
points in $E$. 

As $B_l$ is a zonoid, we can approximate it by finite
sums of segments, namely $B_l=\displaystyle\lim_N \displaystyle\sum_{j=1}^{N}[-v_j,v_j]$, where $v_j\in E$ are so that they span $E$. 
Using the continuity and the Minkowski additivity of the complex difference body operator and Lemma \ref{l:dc segment}, we have
\[
\D_C{B_l}=\D_C{\lim_N{\sum_{j=1}^N{[-v_j,v_j]}}}=\lim_N\sum_{j=1}^N \D_{[-1,1]}{(C\cdot v_{j})}.
\]
Since  $\D_{[-1,1]}{(C\cdot v_{j})}$ is $(iC-iC)\cdot v_{j}\subset \lin_{\C}\{v_j\}$, $C$ has interior points,
and the set of vectors $v_j\in E$ must span $E$, there will be
interior points (relative to $E$) in $\sum_{j=1}^N \D_{[-1,1]}{C\cdot v_{j}}$ and hence in $\D_C{B_l}$
too. Now, it is enough to consider $\inr(K;B_l)\inr(D_C B_l; K)$. Notice, that the dimension of
$D_C B_l$ may be strictly larger than the dimension of $K$.
\end{proof}

Let $K\in\K(W)$ and $u\in S^{2m-1}$. The \emph{width of $K$ in the
direction $u$} is $w(K,u)=h(K,u)+h(K,-u)$. The minimum and maximum
of the widths of $K$ are, respectively, the \emph{minimal width} $w(K)$ of $K$ and the
\emph{diameter} $\diam(K)$ of it. In the next lemma we give bounds for the width of $\dck$.

\begin{lemma}\label{l: width}
Let $K\in\K(W)$ and $C\in\K(\C)$. Then

\begin{enumerate}\itemsep10pt

\item for every $u\in S^{2m-1}$, \[
w(K) \, \l(C)\leq w(\dck,u)\leq \diam(K)\, \l(C) .
\]

\item $w(K) \l(C)\leq w(\dck)\leq \diam(\dck)\leq \diam(K)\l(C)$. 
\end{enumerate}
If $C$ is a segment, we have equality everywhere.
\end{lemma}

\begin{proof}
It follows from the fact that 
\[w(\dck,u)=h(\dck,u)+h(\dck,-u)=\int_{S^1}\left(h(\alpha K,u)+h(\alpha K,-u)\right) d\Sa(C,\alpha).\]
\end{proof}

\begin{remark}
We observe that when constructing the complex difference body of a
convex body $K$ with respect to $C\in\K(\C)$ with $\l(C)=1$, the diameter is not increasing and the width is not decreasing. Furthermore, the mean width remains unchanged. 
When $C$ is one dimensional, i.e., when $\dck$ is
the usual difference body (up to rigid motions), the ratio between
diameter and width remains unchanged. In some sense these facts support the
idea stated in the introduction that the complex difference body {\it should be rounder} than the original $K$.
\end{remark}

As a corollary of Lemma \ref{l: width}, we have in particular the following
result for convex bodies of constant width.

\begin{corollary}\label{cor:symCtnt} Let $K\in\K(W)$ be a convex body of constant width $w(K)$ and $C\in \K(\C)$. Then,
$\dck$ has constant width equal to $\l(C)w(K)$. 

If, in addition, $C$ is symmetric, then $\dck$ is a ball of radius $\l(C)w(K)/2$.
\end{corollary}
\begin{proof}
The first assertion follows straightforward using Lemma \ref{l: width}.

For the second one, if $C$ is symmetric, then $\D_{C}K$ is also
symmetric (cf. Lemma \ref{p: rho C}). In addition, it is well known that the only symmetric convex bodies with constant width are the balls.
\end{proof}

If instead of the maximum and minimum of the widths of a convex
body $K\in\K(\R^n)$, its average is considered, we are dealing (up to constants) with the so called \emph{mean width of $K$}, $b(K)$. Indeed, again up to a constant, $b(K)$ coincides with the $(n-1)$-st quermassintegral of $K\in\K(\R^n)$.
Quermassintegrals are particular cases of the more general mixed
volumes (for which we refer to \cite[Chapter 5]{schneider_book93}) obtained when only one convex body and a ball are involved.

Next we provide bounds for the quermassintegrals, $\W_i(K)$, $i=0,1\dots,2m$ of $\dck$, in terms of the quermassintegrals of $K$ and the length of $C$. Notice that $i=0$ is just the volume and $i=2m$ is the volume of the unit ball, denoted by $\kappa_{2m}$.

\begin{proposition}\label{p: bound quermass}
Let $K\in\K(W)$ and $C\in \K(\C)$. For any $i=0,\dots,2m-1$,
\begin{equation}\label{e: ineq quermass}
\W_i(\dck)\geq \W_i(K)\, \l(C)^{2m-i}.
\end{equation}
\end{proposition}

\begin{proof}
Using the multilinearity of the mixed volumes 
and applying the Brunn-Minkowski inequality and the Minkowski first
inequality we obtain
\[
\begin{split}
\W_i(\dck) & =  \int_{S^1}\dots\int_{S^1}{\V(\alpha_1 K, \dots,
\alpha_{2m-i}K, B_{2m}[i])}d\Sa(C,\alpha_1)\dots
d\Sa(C,\alpha_{2m-i})\\
& \geq \int_{S^1}\dots\int_{S^1}{\prod_{j=1}^{2m-i}\V(\alpha_j
K[2m-i],B_{2m}[i])^{\frac{1}{2m-i}}d\Sa(C,\alpha_1)\dots
d\Sa(C,\alpha_{2m-i})}\\
& = \int_{S^1}\dots\int_{S^1}{\W_i(K)}d\Sa(C,\alpha_1)\dots
d\Sa(C,\alpha_{2m-i})\\
&=\W_i(K) \l(C)^{2m-i},
\end{split}
\]
where we have also used that mixed volumes are invariant under (the same) rigid motions of all their components.
\end{proof}
When $i=2m-1$ we have equality in the above inequality for any
$K\in\K(W)$ and $C\in\K(\C)$, since $\W_{2m-1}$ is linear.
Furthermore, if $\alpha K=K$ for all $\alpha\in S^1$, we will have
equality for any $i=0,1,\dots,2m-1$.  For a fixed convex body $C$, the equality holds for every convex body $K$ satisfying
$h(\alpha_{i}K,u)=h(\alpha_{j}K,u)$ for every $u\in W$ and
$\alpha_{i}$, $\alpha_{j}$ normal vectors to $C$.

As a consequence of \eqref{e: ineq quermass} and Theorem \ref{dimDC}, the following lower bound for the
volume of $\dck$ holds.

\begin{proposition}\label{voldck}Let $K\in\K(W)$ be a convex body of dimension $l$ contained in $E$.
Let $a:=\dim_{\C}(E\cap JE)$. Then, for any $C\in\K(\C)$ we have
$$\l(C)^{d}\vol_{d}(K)\leq\vol_{d}(\D_{C}K),$$
where $d=2(l-a)$.
$S^1$-invariant convex bodies provide equality.
\end{proposition}

Despite the mentioned equality cases, the bound for the volume
contained in the result is certainly bad for convex bodies $K$
which are not $2m$-dimensional and $l\neq 2a$. 

Unfortunately using the above idea we have been able only to bring the dimension of $K$ into play when bounding
a precise quermassintegral of $\dck$. For, we use the
normalization of quermassintegrals known as intrinsic volumes. If,
furthermore, the classical inradius $\inr(K)$, and circumradius, $\cir(K)$, come into play, we
get further lower bounds, as well as upper bounds involving both,
the volume of $\dck$ and the volume (in the corresponding dimension)
of $K$. For completeness we include  the mentioned results in the following
lemma. 

\begin{lemma}
Let $K\in\K(W)$ with $\dim K=l\leq 2m$ and $C\in \K(\C)$. Then
\begin{enumerate}
\item $\W_{2m-l}(\dck)\geq \frac{\kappa_{2m-l}}{\binom{2m}{l}}\l(C)^l
\vol_l(K),$

\item $\vol(\dck)\geq
\frac{\kappa_{2m-l}}{\binom{n}{l}}\inr(\dck)^{2m-l}\l(C)^{l}\vol_l(K),$

\item $\vol(\dck)\geq \frac{\kappa_{2m-l}}{\binom{n}{l}}\inr(K)^{2m-l}\l(C)^{2m}\vol_l(K),$

\item $\vol(\dck)\leq \kappa_{2m}\cir(K)^{2m}\l(C)^{2m}.$
\end{enumerate}
Equality holds in $i)$ and $ii)$ if $l=2m$ and $K=\alpha K$ for all $\alpha\in S^1$, i.e., $K$ is a $S^1$-invariant convex body.
If $K=B_{2m}$ there is equality in $iii)$ and $iv)$.

\end{lemma}

\begin{proof}
Notice that from the definition of inradius and circumradius we have that, up to translation,
\[\inr(\dck)B_{2m}\subseteq\D_{C}K\subseteq R(\D_{C}K)B_{2m}.\]

\begin{enumerate}
\item Using Proposition \ref{p: bound quermass} for the case $i=2m-l$, we have
that $\W_{2m-l}(\dck)\geq \W_{2m-l}(K)\l(C)^{l}$. Since $K$ has dimension $l$, using 
the relation between quermassintegrals and intrinsic volumes (\cite[(4.2.26)]{schneider_book93}) 
we have that 
\[
\W_{2m-l}(K)=\frac{\kappa_{2m-l}}{\binom{2m}{l}}\vol_l(K),
\] which gives $i)$.

\item Since mixed volumes are monotone, we have that
\[
\vol(\dck)\geq \V(\inr(\dck)B_{2m}[2m-l],\dck[l])=\inr(\dck)^{2m-l}\W_{2m-l}(\dck).
\] Then the inequality follows from $i)$.

\item It is enough to prove it for $\dim K=2m$. 
 In this case, $\inr(K)B_{2m}\subseteq K$. Now, $i)$, the monotonicity of the
complex difference body operator and the inradius ensure that $\inr\left(\D_C \inr(K)B_{2m}\right)\subseteq \inr(\dck)$, from which
the second inequality follows by observing that $\inr\left(\D_C \inr(K)B_{2m}\right)=\inr(K)\l(C)$.

\item From the monotonicity of $\D_{C}$ we have $\dck\subseteq\D_{C}(R(K)B_{2m})=\l(C)R(K)B_{2m}$ and the inequality follows.
\end{enumerate}
\end{proof}

An immediate consequence coming from the classical
Brunn-Minkowski inequality, since $\D_{C}$ is Minkowski additive, i.e., $\D_{C}(K+L)=\D_{C}\!K+\D_{C}\!L$, is the following

\begin{corollary}Let $K,L\in\K(W)$. Then
$$\vol(\D_{C}(K+L))^{1/2m}\geq \vol(\D_{C}\!K)^{1/2m}+\vol(\D_{C}\!L)^{1/2m}.$$
\end{corollary}

Next we deal with some very classical magnitudes, such as diameter $\diam$ and circumradius $\cir$ of the complex difference body $\dck$.

\begin{proposition}
Let $K\in\K(W)$ and $C\in\K(\C)$ be a polygon. Then, there is a constant $\mu_{C}>0$ independent of $K$ so that 
$$\mu_{C}\d(K)\l(C)\leq \d(\D_{C}K)\leq \d(K)\l(C),$$
and $$\mu_{C}\sqrt{\frac{2m+1}{m}}R(K)\l(C)\leq R(\D_{C}K)\leq R(K)\l(C).$$
The upper bound holds, in both cases, for every $C\in\K(\C)$.
\end{proposition}
\begin{proof}
The upper bound for $\d(\dck)$ was proved in Lemma \ref{l: width}.
For the lower bound, we use that for $K,K'\in \K(W)$  (see \cite[Theorem 1.2]{G-HC})
$$\sqrt{2}\d(K+K')\geq\d(K)+\d(K').$$
Applying this inequality to
$\D_{C}K=\sum_{i=1}^Nr_{i}\alpha_{i}K$, where $\alpha_{i}\in S^1$ are the normal vectors of $C$ to the edges $s_{i}$ with length $r_{i}$, $i=1,\dots,N$ we have
$$\d(\D_{C}K)\geq \sum_{i=1}^{N-1}\frac{1}{2^{i/2}}\d(\alpha_i r_i K)+\frac{1}{2^{(N-1)/2}}\d(\alpha_N r_N K)\geq \sum_{i=1}^{N}\frac{1}{2^{i/2}}\d(\alpha_i r_i K). $$
Then, taking $s=\sum_{i=1}^N 2^{i/2}$, since the diameter is $1$-homogeneous, we get
\[
\d(\D_{C}K)\geq \frac{\l(C)}{s}\diam(K).
\]

The lower bound of the circumradius, follows from previous inequality for the
diameter and Jung's inequality between circumradius and
diameter (see \cite{Jung}).  
We get
\begin{align*}
\cir(\D_{C}K)&\geq\frac{\d(\D_{C}K)}{2}\geq\frac{\l(C)}{s}\d(K)\geq\frac{\l(C)}{s}\sqrt{\frac{2m+1}{m}}\cir(K).
\end{align*}

We prove the upper bound for $R(\dck)$, for $C$ a polygon. The general result for the upper bound follows by approximation. It holds
\begin{align*}
\cir(\D_{C}K)=\cir(\sum_{i=1}^Nr_{i}\alpha_{i}K)\leq\sum_{i=1}^N\cir(r_{i}\alpha_{i}K)=\sum_{i=1}^Nr_{i}\cir(K)=\l(C)\cir(K).
\end{align*}
\end{proof}

Of course, if $C$ is not a polygon, approximating it by polytopes, $\mu_{C}=0$ and the trivially zero as lower bound is obtained.

To finish this section we include an inequality for the mixed volume of $j$ times  $K$ and $2m-j$ times $\dck$.
For a conjecture on an upper bound in the classical case we refer to \cite[Notes Section 7.3]{schneider_book93}.

\begin{lemma}\label{l: ineq dim 2}
Let $K\in\K(W)$ and $C\in\K(\C)$. Then
\[
\V(K[j],\dck[2m-j])\geq \l(C)^{2m-j} \vol(K).
\]
\end{lemma}

\begin{proof}
The proof follows as a consequence of a special case of the Aleksandrov-Fenchel inequality. Indeed, using the multilinearity of the mixed volumes, we have
\[
\V(K[j],\dck[2m-j])=\int_{S^1}\cdots\int_{S^1}\V(K[j],\alpha_{1}K,\dots,\alpha_{2m-j}K) d\Sa(C,\alpha_{1})\dots d\Sa(C,\alpha_{2m-j}).
\]
From the Aleksandrov-Fenchel inequality, it follows (see \cite[(B.18)]{gardner.book06})  
\[
\V(K[j],\alpha_{1}K,\dots,\alpha_{2m-j}K)^{2m}\geq\vol(K)^{j}\vol(\alpha_{1}K)\cdots\vol(\alpha_{2m-j}K)=\vol(K)^{2m},
\]
which yields straightforward the result.
\end{proof}

\section{Spherical harmonics in $S^{2m-1}$}\label{s: sph harm}
In this section we introduce the notions and results we shall use from the theory of spherical harmonics in the unit sphere of $\C^m$ endowed with the standard hermitian product. We first review the basic notions in a real vector space. 
For a description of the theory of spherical harmonics with applications to convex geometry we refer to the book \cite{groemer.book}. The results for a complex vector space can be found in \cite{koornwinder73,menegatto.oliveira.orthogonal,quinto,rudin.unitball.book} and references therein.

Let $\R^n$ be endowed with the usual scalar product. A \emph{spherical harmonic on $S^{n-1}$} is the restriction to $S^{n-1}$ of a harmonic polynomial $p$ on $\R^n$, i.e., $\Delta p=0$ where $\Delta$ denotes the Laplace operator on $\R^n$. It is said \emph{of degree $k$} if it is 
the restriction of a harmonic polynomial of degree of homogeneity $k$ with $k\in\N$. In the following we assume that $0\in\N$. 
We denote by $\hH^n_{k}$ the vector space of spherical harmonics in $S^{n-1}$ of degree $k$ and by $\hH^n$ the vector space of the spherical harmonics in $S^{n-1}$. It is known that $\hH^n$ can be decomposed as 
$$\hH^n=\bigoplus_{k=0}^{\infty}\hH_{k}^n,$$
and $\dim \hH^n_{k}=:N(n,k)=\frac{n+2k-2}{n+k-2}\binom{n+k-2}{n-2}$. Here and in the following, the orthogonality is understood with respect to the standard product on $L^{2}(S^{n-1})$:
$$\langle f,g\rangle_{L^2}:=\int_{S^{n-1}}f(u)g(u)d\sigma_{n-1}(u),\quad f,g\in L^{2}(S^{n-1}),$$
where $d\sigma_{n-1}$ denotes the standard measure on the sphere $S^{n-1}$.

If $f\in L^{2}(S^{n-1})$ and $\{Y_{kj}\in\hH_{k}^n\,:\,k\in\N,\,j=1,\dots,N(n,k)\}$ is an orthonormal system, then the \emph{harmonic expansion of $f$} with respect to this orthonormal system is given by
$$f\sim \sum_{k=0}^{\infty}\sum_{j=1}^{N(n,k)}\langle f,Y_{kj}\rangle Y_{kj}=\sum_{k=0}^{\infty}\pi_{k}f,$$
where $\pi_{k}f:=\sum_{j=1}^{N(n,k)}\langle f,Y_{kj}\rangle Y_{kj}$ is the orthogonal projection of $f$ onto $\hH_{k}^n$.

We observe that the spherical harmonics of even (respectively odd) degree are even (respectively odd) functions on the sphere.
Moreover, the space $\hH^n_{k}$ is rotation-invariant (i.e., if $Y\in\hH^n_{k}$, then $\varphi Y\in\hH^n_{k}$ for every $\varphi\in\SO(n)$) and irreducible, i.e., the only subspaces of $\hH^n_{k}$ which are $\SO(n)$-invariant are the trivial and the total space $\hH^{n}_{k}$. Fixing a point $v_{0}\in S^{n-1}$, we can describe $S^{n-1}$ as the homogeneous space $$\SO(n)/\SO(n-1)$$ where $\SO(n-1)$ is the isotropy group of $v_{0}$. Using this description of $S^{n-1}$ and the irreducibility of $\hH^n_{k}$, $k\in\N$, we get from the representation theory that there exists a unique spherical harmonic $Q_{k}$ of degree $k$, $\SO(n-1)$-invariant and such that $Q_{k}(v_{0})=1$. The polynomials $\{Q_{k}\}_{k}$ define the \emph{Legendre polynomials}, which are essential in the theory of spherical harmonics. 

In our situation, we have $\C^m\cong\R^{2m}$ endowed with the standard hermitian product $(\cdot, \cdot)$ as an ambient space. Considering the sphere $S^{2m-1}$ in $\R^{2m}$, all the previous facts remain valid. However, as a vector space over $\C$, the multiplication by complex numbers gives rise to a richer structure on the space $\hH^{2m}$.

\begin{definition}A spherical harmonic $Y\in\hH^{2m}$ is \emph{homogeneous of bi-degree $(k,l)$} if $Y\in\hH^{2m}_{k+l}$ and 
$$Y(\alpha u)=\alpha^k\overline{\alpha}^{l}Y(u)\quad\forall\alpha\in S^1,\,\,u\in S^{2m-1}.$$
We denote by $\hH^{2m}_{k,l}$ the space of spherical harmonics on $S^{2m-1}$ of bi-degree $(k,l)$. 
\end{definition}
It holds
\begin{equation}\label{orthHk}\hH^{2m}_{k}=\bigoplus_{l=0}^k\hH^{2m}_{l,k-l}\quad\forall k\in\N.\end{equation}
Each space $\hH^{2m}_{l,k-l}$ is invariant under the unitary group $\U(m)$ and also irreducible (see \cite{koornwinder75}). Thus, a function $f\in\mathcal{C}(S^{2m-1})$ can be expressed as 
$$f\sim \sum_{k=0}^{\infty}\sum_{l=0}^k\pi_{l,k-l}(f),$$
where $\pi_{k,l}(f)$ is the orthogonal projection of $f$ onto $\hH^{2m}_{k,l}$, and $f\in\cC(S^{2m-1})$ is univocally determined by $\pi_{k,l}(f)$, $(k,l)\in\N^2$.
Fixed a point $e\in S^{2m-1}$ with isotropy group $\U(m-1)$, we have $S^{2m-1}\cong \U(m)/\U(m-1)$. Similar to the real case, we get that for every $(k,l)\in\N^2$ there exists a unique spherical harmonic $P_{k,l}$ of degree $k+l$, $\U(m-1)$-invariant and such that $P_{k,l}(e)=1$. These polynomials, analogues to the Legendre polynomials, are the so-called \emph{Jacobi polynomials}.

\begin{definition}A function $\tilde P\in\mathcal{C}(S^{2m-1})$ is called \emph{spherical function of order $(k,l)$} if $\tilde{P}\in\hH_{k,l}^{2m}$, $\tilde{P}$ is $\U(m-1)$-invariant and $\tilde{P}(e)=1$. Its associated \emph{Jacobi polynomial $P:B_2\to\C$ of order $(k,l)$} is defined by $P(z)=\tilde{P}(w)$, where $w\in S^{2m-1}$ is such that $(e,w)=z$ and $B_{2}$ denotes the unit disc in $\C$. 
\end{definition}
Jacobi polynomials satisfy the following properties. 
\begin{proposition}[\cite{johnson.wallach,quinto}]\label{propPkl}For every $k,l\in\N$ the vector space $\hH_{k,l}^{2m}$ contains a unique spherical function given by $P_{k,l}((e,\cdot))$ and satisfying 
\begin{align*}
&P_{k,l}(z)=\overline{P_{k,l}}(\overline z), \quad k,l\in\N,
\\ &P_{k,l}(z)=r^{|k-l|}e^{i\theta(k-l)}Q_{\min\{k,l\}}(|k-l|,m-2,|z|^2),\,\,\, k,l\in\N,\, z=re^{i\theta}, r\geq 0, \theta\in[0,2\pi),
\end{align*}
where $\{Q_{l}(a,b,t)\,:\,l\in\N\}$ constitutes a complete set of polynomials in $t\in\R$ orthogonal on $[0,1]$ in weight $t^{a}(1-t)^bdt$ and satisfying $Q_{l}(a,b,1)=1$, $a,b>-1$.
\end{proposition}
The set $\{P_{k,l}\}_{k,l}$ constitutes a complete orthogonal set of functions in $L^2(B_2)$.

One of the main results in the theory of spherical harmonics is the Funk-Hecke theorem. A complex version of it was proved in \cite{quinto}.
We shall denote by $L^2(B_2,((1-|z|^{2})^{m-2})dz)$ the weighted $L^2$ space on $B_{2}$ with respect to the density measure $((1-|z|^{2})^{m-2})dz$.

\begin{theorem}[Complex Funk-Hecke Theorem, \cite{quinto}] Let $K:B_2\to\R$ be in $L^2(B_2,((1-|z|^{2})^{m-2})dz)$ and $Y_{k,l}\in\hH^{2m}_{k,l}$. Then, 
$$\int_{S^{2m-1}}K((u,v))Y_{k,l}(v)d\sigma_{2m-1}(v)=\lambda_{k,l}Y_{k,l}(u),$$
with
$$\lambda_{k,l}=\sigma_{2m-3}\int_{B_2}K(z)\overline{P_{k,l}}(z)(1-|z|^2)^{m-2}dz,$$
where $\sigma_{2m-3}$ denotes the surface of $S^{2m-3}$. 
\end{theorem}

\begin{definition}A linear map $A:\cC(S^{2m-1})\to\cC(S^{2m-1})$ is a \emph{multiplier transformation}, if for every $k,l\in\N$ there exists a complex number $\lambda_{k,l}$ such that
$$\pi_{k,l}(Af)=\lambda_{k,l}\pi_{k,l}(f)\quad\forall f\in\cC(S^{2m-1}).$$
The complex numbers $\lambda_{k,l}$ are called \emph{multipliers of the map $A$}. 
\end{definition}

Similar to the real case, a multiplier transformation is injective if and only if $\lambda_{k,l}\neq 0$ for every $k,l\in\N$ (cf. \cite{menegatto.oliveira}). 

Note that if the linear map $A:\cC(S^{2m-1})\to\cC(S^{2m-1})$ is given by a kernel $K:B_2\to\R$ in $L^2(B_2,((1-|z|^{2})^{m-2})dz)$, then by the complex Funk-Hecke theorem we have that $A$ is a multiplier transformation. 
In the following, we shall use the harmonic expansion of the support function of a convex body, seen as a function on the sphere $S^{2m-1}$.

\vspace{0.2cm}
In the previous lines, we have introduced the theory of spherical harmonics for the space of continuous functions on the sphere. However, this theory can be established for its dual space, the space of Borel measures on the sphere. In this paper, we shall only treat Borel measures on $S^1$, for which we can consider the classical Fourier series expansion without using spherical harmonics. We identify $S^1$ with $[0,2\pi]$ and the functions on $S^1$ with $2\pi$-periodic functions on $\R$. In the identification of $S^1$ with $[0,2\pi]$, we denote the elements of both spaces by $\alpha$, but it should be clear from the context where $\alpha$ is meant. Using the standard notation, we define the \emph{$j$-th Fourier coefficient $c_{j}(f)$ of $f\in\cC(S^1)$} and, analogously, \emph{of a Borel measure $\mu$} on $S^1$ by
\begin{align*}c_{0}(f)&:=\frac{1}{2\pi}\int_{0}^{2\pi}f(\alpha)d\alpha,\quad c_{0}(\mu):=\frac{1}{2\pi}\int_{0}^{2\pi}d\mu(\alpha),
\\c_{j}(f)&:=\frac{1}{\pi}\int_{0}^{2\pi}e^{ij\alpha}f(\alpha)d\alpha,\quad c_{j}(\mu):=\frac{1}{\pi}\int_{0}^{2\pi}e^{ij\alpha}d\mu(\alpha),\quad j\in\mathbb{Z}.
\end{align*}
We shall use the following fact for the surface area measure of a planar convex body.
\begin{lemma}\label{relFour}Let $C\in\K(\R^2)$. The Fourier coefficients of the surface area measure of $C$, $\Sa_{C}$, are related with the Fourier coefficients of the support function of $C$ by
$$c_{j}(\Sa_{C})=(1-j^2)c_{j}(h_{C}),\quad j\in\mathbb{Z}.$$
\end{lemma}
\begin{proof}By definition, 
$$c_{j}(\Sa_{C})=\frac{1}{\pi}\int_{0}^{2\pi}e^{ij\alpha}d\Sa(C,\alpha)=\frac{1}{\pi}\int_{0}^{2\pi}e^{ij\alpha}(h_{C}(\alpha)+h_{C}''(\alpha))d\alpha.$$
Interpreting $h_{C}+h_{C}''$ as a distribution and $\alpha\mapsto e^{ij\alpha}$ as a $2\pi$-periodic test function on $\R$, we can use twice integration by parts and we obtain the result.
\end{proof}

\begin{remark}
Note that $c_{-j}(h_{C})=\overline{c_{j}(h_{C})}$ since $h_{C}$
takes only real values. 
\end{remark}

Next, we collect the main geometric cases we will treat. For, we need the following lemma.

\begin{lemma}\label{coeffS1}A function $f\in L^{2}(S^{2m-1})$ is $S^1$-invariant, i.e., $f(\alpha u)=f(u)$ for every $\alpha\in S^1$, $u\in S^{2m-1}$, if and only if its harmonic expansion is given by
$$f\sim \sum_{j=0}^{\infty}\pi_{j,j}(f).$$
\end{lemma}
\begin{proof}From the definition of the spherical harmonics of bi-degree $(k,l)$, it follows directly that $Y\in\hH^{2m}_{k,l}$ is a $S^1$-invariant function on $S^{2m-1}$ if and only if $k=l$. 
\end{proof}

\begin{examples}\label{examples}
\hfill
\begin{enumerate}\itemsep11pt
\item \label{coeffBalls}\emph{Balls.} It is well known that $\pi_{j}(h_{B_{2m}})=0$ for every $j\in\N$, $j>0$. 
From the orthogonal decomposition in \eqref{orthHk} it also holds $\pi_{k,l}(B_{2m})=0$ for every $(k,l)\in\N^2 \backslash\{(0,0)\}$.

\item \label{coeffRtheta}\emph{$S^1$-invariant convex bodies.} Let $K\in\K(\C^m)$ be a $S^1$-invariant convex body. Then, its support function is invariant under the action of $S^1$, i.e., $$h(K,\alpha u)=h(K,u)\quad\forall \alpha\in S^1,\, u\in S^{2m-1}.$$
Thus, using Lemma \ref{coeffS1}, we have that the harmonic expansion of the support function of a $S^1$-invariant convex body only contains terms of the form $\pi_{j,j}(h_{K})$, $j\in\N$.

\item \label{coeffosim}\emph{Symmetric convex bodies.} The Fourier coefficients $c_{j}(h_{C})$ of a symmetric convex body $C\in\K(\R^2)$ satisfy $c_{2j+1}(h_{C})=0$, for every $j\in\Z$ (since $h_{C}(\alpha+\pi)=h_{C}(\alpha)$ for every $\alpha\in[0,2\pi]$). 

In higher dimensions, for a symmetric convex body $K$, it holds $\pi_{2j+1}(h_{K})=0$ since $h_{K}(-u)=h_{K}(u)$ for every $u\in S^{2m-1}$. Hence, 
we also have $\pi_{k,l}(h_{K})=0$ for every $(k,l)\in\N^2$ such that $k+l$ is odd.

\item \label{coeffConsWidth}\emph{Convex bodies of constant width.} 
For a convex body $K$ of constant width, it holds $\pi_{2j}(h_{K})=0$, $j>0$ (see \cite[Proposition 5.7.1]{groemer.book}). If $K\in\K(\C^m)$ has constant width, then $\pi_{k,l}(h_{K})=0$ for every $(k,l)\in\N^2 \backslash\{(0,0)\}$ such that $k+l$ is even.

\item  \label{coeffI} \emph{Interval.} Assume that $I\subset\C$ is the interval $[-i/2,i/2]$. 
One can easily compute the Fourier coefficients of $\Sa_{I}$ obtaining
$c_{0}(\Sa_{I})=1/\pi$ and $c_{j}(\Sa_{I})=(1+(-1)^{j})/\pi$. Notice that the perimeter of a segment in $\R^2$ is twice its length.

\item \emph{Difference body.} The coefficients of the harmonic expansion of the difference body $D K$ 
are given by $$\pi_{j}(h_{D K})=\pi_{j}(h_{K})+\pi_{j}(h_{-K})=\left\{\begin{array}{rl} 2\pi_{j}(h_{K}), & \textrm{if } $j$ \textrm{ is even,}
\\ 0,  & \textrm{if } $j$ \textrm{ is odd.}
\end{array}\right.$$
Hence, the difference body operator (extended from the space of support functions to the space of continuous functions on the sphere) is a multiplier transformation with multipliers $\lambda_{2j}=2$ and $\lambda_{2j+1}=0$, $j\in\N$, which coincides with the Fourier coefficients of $\Sa_{I}$ up to its normalization factor $1/2\pi$ or $1/\pi$.
\end{enumerate}
\end{examples}

\section{The operator $\D_{C}$ as a multiplier transformation}\label{s:multiplier}

In view of the last example above, we would like to interpret the operator $\D_{C}$, for a general $C\in\K(\C)$, as a multiplier transformation and express its multipliers in terms of the Fourier coefficients of $h_{C}$. This is the content of the next result. 
\begin{theorem}\label{t1}Let $C\in\K(\C)$ be a convex body. 
The extension of the operator $\D_{C}$ to the space of continuous functions on $S^{2m-1}$ given by
$$\func{A_{C}}{\cC(S^{2m-1})}{\cC(S^{2m-1})}{f}{\int_{S^{1}}f(\overline{\alpha}\cdot)d\Sa(C,\alpha)}$$
is a multiplier transformation and the multipliers are 
\begin{equation}\label{multipliers}\lambda_{k,l}(A_{C})=\int_{0}^{2\pi}e^{-i(k-l)\alpha}d\Sa(C,\alpha).\end{equation}
\end{theorem}
\begin{proof}
Assume first that $C\in\K(\C)$ is of class $\cC^{2}$, that is, its support function is of class $\cC^{2}$. In this case, we have that $g:=h_{C}+h_{C}''$ is a continuous function on $S^1$. We denote by $\check g$ its homogeneous extension of degree 0 on $\R^{2}$.
Let $\epsilon>0$ and $v\in S^{2m-1}$. We define the map $K_{\epsilon,g}((v,\cdot))$ on $S^{2m-1}$ by
$$K_{\epsilon,g}((v,\cdot)):=\frac{m-1}{2^{m-2}\sigma_{2m-3}\epsilon^{m-1}}\chi_{[1-\epsilon,1]}(|(v,\cdot)|)\check{g}((v,\cdot)).$$
We claim that the function
$$v\mapsto\lim_{\epsilon\to 0}\int_{S^{2m-1}}K_{\epsilon,g}((w,v))f(w)d\sigma_{2m-1}(w)$$
converges uniformly to $A_{C}f$.

For $v\in S^{2m-1}$ fixed, any vector $w\in S^{2m-1}$ can be uniquely decomposed as  
$$w=z v+\sqrt{1-|z|^2}\xi,\quad z\in B_{2},\,\,\xi\in S^{2m-1},\,(\xi,v)=0.$$
The parametrization of the sphere corresponding to this decomposition, is given by 
$$d\sigma_{2m-1}(w)=(1-|z|^2)^{m-2}dzd\sigma_{2m-3}(\xi),$$
where $d\sigma_{j}$ is the standard measure on the sphere $S^j$.
Then, we have 
$$\int_{S^{2m-1}}\!\!\!\!K_{\epsilon,g}((v,w))f(w)d\sigma_{2m-1}(w)\!=\!
\!\int_{S^{2m-3}}\!\int_{B_2}\!\!\!K_{\epsilon,g}(z)f(zv+\sqrt{1-|z|^2}\xi)(1-|z|^2)^{m-2}dzd\sigma_{2m-3}(\xi).$$
Using the definition of $K_{\epsilon,g}$, polar coordinates on $B_2$, the 0-homogeneity of $\check{g}$ and the change of variable $t=(1-r)/\epsilon$, we get 
\begin{align*}&\int_{S^{2m-1}}K_{\epsilon,g}((v,w))f(w)d\sigma_{2m-1}(w)
\\&=\frac{m-1}{2^{m-2}\sigma_{2m-3}\epsilon^{m-1}}\!\int_{S^{2m-3}}\int_{B_2}\!\chi_{[1-\epsilon,1]}(|z|)\check{g}(\overline z)(1-|z|^2)^{m-2}f(zv+\sqrt{1-|z|^2}\xi)dzd\sigma_{2m-3}(\xi)
\\&=\frac{m-1}{2^{m-2}\sigma_{2m-3}\epsilon^{m-1}}\!\int_{S^{2m-3}}\int_{0}^{2\pi}\int_{1-\epsilon}^{1}(1-r^2)^{m-2}rf(re^{i\theta}v+\sqrt{1-r^2}\xi)g(e^{-i\theta})drd\theta d\sigma_{2m-3}(\xi)
\\&=\frac{m-1}{2^{m-2}\sigma_{2m-3}}\int_{S^{2m-3}}\int_{0}^{2\pi}\int_{0}^1t^{m-2}f_{\epsilon}(t,\theta,\xi)g(e^{-i\theta})dtd\theta d\sigma_{2m-3}(\xi),
\end{align*}
where
$$f_{\epsilon}(t,\theta,\xi):=(2-t\epsilon)^{m-2}(1-t\epsilon)f((1-t\epsilon)e^{i\theta}v+\sqrt{t\epsilon(2-t\epsilon)}\xi).$$
For $0<\epsilon<1$ we have $|f_{\epsilon}(t,\theta,\xi)|\leq 2^{m-2}\|f\|_{\infty}<+\infty$ since $0\leq t\leq 1$ and $f$ is bounded on $S^{2m-1}$. Moreover, 
$$\lim_{\epsilon\to 0}f_{\epsilon}(t,\theta,\xi)=2^{m-2}f(e^{i\theta}v).$$
Applying the theorem of the dominated convergence, we get
\begin{align*}\lim_{\epsilon\to 0}\int_{S^{2m-1}}K_{\epsilon,g}((v,w))f(w)d\sigma(w)
&=\frac{m-1}{\sigma_{2m-3}}\!\int_{S^{2m-3}}\int_{0}^{2\pi}\int_{0}^1t^{m-2}f(e^{i\theta}v)g(e^{-i\theta})dtd\theta d\sigma_{2m-3}(\xi)
\\&=\int_{0}^{2\pi}f(e^{i\theta}v)d\Sa(\overline{C},\theta)=\int_{0}^{2\pi}f(e^{-i\theta}v)d\Sa(C,\theta),
\end{align*}
and the claim follows. 

If $C\in\K(\C)$ is not of class $\cC^2$, 
it can be approximated by convex bodies of class $\cC^2$ and, by continuity, we get the result.  

From the complex Funk-Hecke theorem, we have that $A_{C}$ is a multiplier transformation, and the multipliers are given by
$$\lambda_{k,l}(A_{C})=\sigma_{2m-3}\lim_{\epsilon\to 0}\int_{B_2}K_{\epsilon,g}(z)\overline{P_{k,l}}(z)(1-|z|^2)^{m-2}dz.$$
Using polar coordinates on $B_{2}$, Proposition \ref{propPkl}, the 0-homogeneity of $\check{g}$, again the change of variable $t=(1-r)/\epsilon$, the dominated convergence theorem and $Q_{l}(a,b,1)=1$, $a,b>-1$, we get
\begin{align*}
&\lambda_{k,l}(A_{C})\!=
\lim_{\epsilon\to 0}\frac{c}{\epsilon^{m-1}}\!\int_{1-\epsilon}^{1}\!\!r^{|k-l|+1}Q_{\min\{k,l\}}(|k-l|,m-2,r^2)(1-r^2)^{m-2}dr\!\!\int_{0}^{2\pi}\!\!e^{-i\theta(k-l)}d\Sa(C,\theta)
\\&\!=c\lim_{\epsilon\to 0}\!\int_{0}^1\!\!Q_{\min\{k,l\}}(|k-l|,m-2,(1-t\epsilon)^2)t^{m-2}(1-t\epsilon)^{|k+l|+2}(2-t\epsilon)^{m-2}dt\!\!\int_{0}^{2\pi}\!\!\! e^{-i\theta(k-l)}d\Sa(C,\theta)
\\&\!=\int_{0}^{2\pi}e^{-i\theta(k-l)}d\Sa(C,\theta),
\end{align*}
where $c=(m-1)2^{2-m}$. Hence, the result follows. 
\end{proof}

\begin{remark}
The multipliers $\lambda_{k,l}(A_{C})$ can be easily related with the Fourier coefficients of the area measure of $C$ as
$$\lambda_{k,k}(A_{C})=2\pi c_{0}(\Sa_{C}),\quad\lambda_{k,l}(A_{C})=\pi \overline{c_{k-l}}(\Sa_{C}),\,\,k\neq l.$$ 
In order to simplify the notation we will redefine 
$$c_{k-l}(\Sa_{C}):=\lambda_{k,l}(A_{C}),\quad k,l\in\N.$$
From now on $c_{j}(\Sa_{C})$ has to be interpreted as the $j$-th Fourier coefficient of $\Sa_{C}$ with the normalization described in this remark. 
The same renormalization will be considered for the Fourier coefficients of the support function $h_{C}$.
\end{remark}

Now, using Lemma \ref{relFour} we can explicitly give the coefficients
of the harmonic expansion of the support function of the complex difference body, 
as contained in the next result.

\begin{corollary}\label{corCoeff}
The coefficients of the harmonic expansion of the support function of $\D_{C}K$, $C\in\K(\C)$, $K\in\K(\C^m)$ are given by 
$$\pi_{k,l}(h_{\D_{C}\!K})=c_{k-l}(\Sa_{C})\pi_{k,l}(h_{K})=(1-(k-l)^2)c_{k-l}(h_{C})\pi_{k,l}(h_{K}).$$
\end{corollary}

\begin{remark}For every $C\in\K(\C)$ and $K\in\K(\C^m)$, the convex body $\D_{C}K$ has Steiner point at the origin. Indeed, the Steiner point $s(L)$ of a convex body $L\in\K(\C^{m})$ is given by the first coefficient of the harmonic expansion of its support function, more precisely, by (see \cite[p. 430]{schneider_book93})
$$\pi_{1}(h_{L})=\langle s(L),\cdot\rangle_{\R^{2m}},$$
since $s(L)=\frac{1}{\omega_{2m}}\int_{S^{2m-1}}h_{K}(u)ud\sigma(u)$. From Corollary \ref{corCoeff} we have the claim. 
\end{remark}

\section{Fixed points of $\D_{C}K$. Connections of $\D_{C}K$ to classes of convex bodies}\label{s: fixed points}

It is well-known that $D K=2K$ if and only if $K$ is a symmetric convex body. Next, we study the set of fixed points of the operator $\D_{C}$, $C\in\K(\C)$.

\begin{theorem}\label{fixedPoints}Let $C\in\K(\C)$ and $K\in\K(\C^m)$. Then, $\D_{C}K=\lambda K$ if and only if $\lambda=\l(C)$ and either $\pi_{k,l}(h_{K})=0$ or $c_{k-l}(\Sa_{C})=\lambda$, $\forall\, (k,l)\in\N^2$.
\end{theorem}
\begin{proof}We have
$$c_{0}(\Sa_{C})=\int_{S^1}d\Sa(C,\alpha)=\l(C).$$
Using Corollary \ref{corCoeff} we have $\pi_{0,0}(h_{\D_{C}K})=c_{0}(\Sa_{C})\pi_{0,0}(h_{K})$. Hence, $\pi_{0,0}(h_{\D_{C}K})=\lambda\pi_{0,0}(h_{K})$ if and only if $\lambda=\l(C)$.

The second condition also follows directly from Corollary \ref{corCoeff}.
\end{proof}

Notice that if $\pi_{0,0}(h_{K})=0$, then $K$ is a point and $\D_{C}K=0$, for every $C\in\K(\C)$.

From the previous result, we recover the classical case (cf. \ref{examples}.v) and we have these two particular facts. 
\begin{corollary}Let $K$ be a $S^1$-invariant convex body in $\C^m$. Then, $\D_{C}K=\l(C)K$ for every $C\in\K(\C)$.

Let $C$ be the unit disc in $\C$. Then, $\D_CK=2\pi K$ if and only if $K$ is $S^1$-invariant.
\end{corollary}

A question related to the study of the fixed points of an operator concerns the study of the iteration of the operator, i.e., in our case, the study of $\D_{C}^N$ for $N\geq 1$. For the classical case, we clearly have $D\circ \dots\circ D=:D^N=2^{N-1}D$, that is, the iteration of the operator $D$ does not change the ``shape'' of the image of  $K\in\K(\R^n)$. 
From Theorem \ref{fixedPoints} we have that this fact generalizes only for few $C\in\K(\C)$; for instance, if $C$ is a ball.

\begin{corollary}Let $C\in\K(\C)$. Then, $\D_{C}^N=\l(C)^{N-1}\D_{C}$ if and only if for every $j\in\N$, $c_{j}(\Sa_{C})^{N-1}=\l(C)^{N-1}$ or $c_{j}(\Sa_{C})=0$.
\end{corollary}
\begin{proof}Let $K\in\K(\C^m)$. From Theorem \ref{fixedPoints}, we have that $\D_{C}^NK=\l(C)^NK$ if and only if $\pi_{k,l}(h_{K})=0$ or $c_{k-l}(\Sa_{C})=\l(C)$ for every $(k,l)\in\N^2$. Indeed, using recursively that $\D_{C}^NK=\D_{C}(\D_{C}^{N-1}K)$, we have $$\pi_{k,l}(h_{\D_{C}^NK})=c_{k-l}(\Sa_{C})\pi_{k,l}(h_{\D_{C}^{N-1}K})=c_{k-l}^{N}(\Sa_{C})\pi_{k,l}(h_{K}),$$
and the claim follows. 
\end{proof}

\smallskip
We study next the set of fixed points of $\D_{C_{N}}\!$ being $C_{N}$ a regular polygon with $N$ sides. This case enlightens how the ``symmetry'' of $C_{N}$ is inherited by $\D_{C_{N}}\!K$. 

\begin{proposition}
Let $C_{N}\subset\C$ be a regular polygon with $N$ sides each of length $1/N$ 
and normal vector $\alpha_i$, $i=\{1,\dots,N\}$. Then,
$$\D_{C_{N}}\!K=K\Leftrightarrow \alpha_i K=K \quad\forall i\in\{1,\dots,N\}.$$
\end{proposition}
\begin{proof}
Using that $\D_{C}K$ is invariant under translations of $C$, we can assume that $C_{N}$ is a regular polygon centered at the origin. Hence, it has rotational symmetry of order $N$, 
and we have that the normal vectors, as elements in $S^1$, constitute a finite group $G_N$, 
i.e., for every $i,j\in\{1,\dots,N\}$ we have $\alpha_i\alpha_j=\alpha_k$ for some $k\in\{1,\dots,N\}$.
Here the multiplication on the group $G_N$ is given by the usual product of complex numbers.

We claim that $\alpha_i \D_{C_{N}}\!K=\D_{C_{N}}\!K$ for every $K\in\mathcal{K}(\C^m)$ and every $i\in\{1,\dots,N\}$.
Indeed, 
\begin{align*}
h(\alpha_i \D_{C_{N}}K,u)&=\int_{S^1}h(\alpha K,\overline{\alpha_i}u)d\Sa(C_{N},\alpha)
=\int_{S^1}h(\alpha_i\alpha K,u)d\Sa(C_{N},\alpha)
\\&=\frac{1}{N}\sum_{j=1}^N h(\alpha_i\alpha_jK,u)
=\frac{1}{N}\sum_{k=1}^N h(\alpha_kK,u)
=h(\D_{C_{N}}\!K,u).
\end{align*}

Thus, $\D_{C_{N}}\!K=K$ implies that $\alpha_iK=K$ for every $i=\{1,\dots,N\}$. If $K$ satisfies 
the previous condition, then $$h(\D_{C_{N}}\!K,u)=\frac{1}{N}\sum_{j=1}^N h(\alpha_i K,u)=h(K,u)$$
and the result follows.
\end{proof}

Next we study different \emph{families} of convex bodies, such as  Minkowski classes, rotation means and universal convex bodies, in relation
to the complex difference body.

\begin{definition}A \emph{Minkowski class of convex bodies in $\R^n$} is a subset $\mathcal{M}\subset\K(\R^n)$ closed in the Hausdorff metric and closed under Minkowski addition and nonnegative dilations.

Let  $G$ be a group acting on the set of convex bodies. A Minkowski class $\mathcal{M}$ is \emph{$G$-invariant} if for every $K\in\mathcal{M}$ and $g\in G$ we have $gK\in\mathcal{M}$.
\end{definition}

Examples of Minkowski classes are the symmetric convex bodies and the zonoids (see \cite{schneider.schuster07} and \cite[p.164]{schneider_book93}). Since the complex difference body operator is linear (for every $C\in\K(\C)$), we get the following result. 

\begin{proposition}For every $C\in\K(\C)$, the set $\mathcal{M}_{C}:=\image\D_{C}\subset\K(\C^m)$ is a Minkowski class. Moreover, if $G$ is a group acting on $\C$ under which $C$ is invariant (i.e.,\! $gC=C$ for every $g\in G$), then $\mathcal{M}_{C}$ is $G$-invariant, where the action of $G$ on $W$ is the diagonal action described in the introduction. 
\end{proposition}

For $C=I$ an interval, the class $\mathcal{M}_{I}$ coincides with the class of symmetric convex bodies. If $C=B_2$ is a centered ball, then $\mathcal{M}_{B_2}$ coincides with the class of $S^1$-invariant convex bodies (cf. Section \ref{rtinv}).

\begin{definition}[\cite{schneider742}] A convex body $K\in\K(\R^n)$ is said to be \emph{universal} if $\pi_{j}(h_{K})\neq 0$ for every $j\in\N$, $j\neq 1$.
\end{definition}

We define a convex body $K\in\K(\C^m)$ to be \emph{complex universal} if $\pi_{k,l}(h_{K})\neq 0$ for every $(k,l)\in\N^{2}$, $|k-l|\neq 1$.
From Corollary \ref{corCoeff} we get the following fact. 
\begin{corollary}
$\D_CK$ is complex universal if and only if $K$ and $C$ are both complex  universal convex bodies.
\end{corollary}

\begin{definition}[\cite{schneider_book93}] For a convex body $K\in\K(\R^n)$, we say that \emph{$K'$ is a rotation mean of $K$} if there are $N\geq 1$, $N\in\N$, and $\rho_i\in \SO(n)$, $i=1,\dots,N$ rotations so that
\[
K'=\frac{1}{N}\left(\rho_1 K +\cdots +\rho_N K\right).
\]
\end{definition}
Hadwiger proved that for any convex body $K$ with positive dimension there is a sequence of rotation means of $K$ converging to a ball (see \cite{hadwiger}).

If $C\in\K(\C)$ is a polygon whose edges have length one, then, up to a constant, the complex difference body $\dck$ is a rotation mean of $K$.
As Corollary \ref{c: dck ball} shall show, rotations in $\SL(W,\C)$ are not enough in order $\dck$ to be a ball. Indeed, this
result characterizes the pairs $K\in\K(\C^m)$ and $C\in\K(\C)$ for which this can be done.

\section{Surjectivity and injectivity of $\D_{C}$ and related questions}\label{rtinv}

In this section, we use the interpretation of the operator $\D_{C}$ as a multiplier transformation in order to obtain statements about its range of injectivity and image for some geometrically interesting cases of $C$ or restricted to some classes of convex bodies. 

\begin{definition}[\cite{schneider_book93}] A non-empty convex body $K\subset\R^n$ is said to be \emph{indecomposable} if $K=K_{1}+K_{2}$ with $K_{1},K_{2}\in\K(\R^n)$ implies that $K_{1}$ and $K_{2}$ are homothetic to $K$. 
\end{definition}

Now we can prove the following 

\begin{corollary}
The map $\D\!: \K(\C)\times \K(\C^m)\to \K(\C^m)$ is neither injective
nor surjective.
\end{corollary}

\begin{proof}
The map $D$ is not injective since for the real case, i.e., $C$ a line segment, it is not.

For the surjectivity, let us consider, in the plane, a triangle $T$, all whose edges have different lengths.
It is known that such triangles are indecomposable (cf.\! \cite[Theorem 3.2.11]{schneider_book93}), i.e., there exist no convex bodies $K,L$ so that $K+L=T$ and $K,L$ are
no dilates of $T$.
Since we know, from Theorem \ref{DcK polytope}, that if $\dck = T$, then both $C,K$ are polytopes, we have that
\[
T=\displaystyle\sum_{i=1}^{N}{\alpha_i r_i K}
\]
where $\alpha_i$ are the unit normal vectors to the edges of $C$, with corresponding lengths $r_{i}$, $i=1,\dots,N$.
Since $T$ is indecomposable $\alpha_i r_i K= a_i T$ for $a_i\geq 0$, $i=1,\dots, N$.
Then, $K$ is, up to a dilation, a rotation of $T$, namely $K=\overline{\alpha_i}r_i^{-1}a_i T$, what yields
\[
T=\displaystyle\sum_{i=1}^{N}{\alpha_i r_i \overline{\alpha_i}r_i^{-1}a_i T}=\displaystyle\sum_{i=1}^{N}{a_i T}.
\]
But this is not $\D_C T$ for any $C\in\K(\C)$ since all the edges of $T$ have different length.
\end{proof}

Next we deal with \textbf{convex bodies of constant width}. We denote by $\mathcal{W}^{2m} \subset \K(\C^m)$ the class of convex bodies of constant width.
 
\begin{corollary}\label{cor:cw}Let $C\in\K(\C)$ and $K\in\K(\C^m)$. Then, $\D_{C}K\in \mathcal{W}^{2m}$ if and only if  for every $(k,l)\in\N^2$ with $k+l\geq 2$ even, either $c_{k-l}(h_{C})=0$ or $\pi_{k,l}(h_{K})=0$.

In particular, $\image\D_{C}\subseteq \mathcal{W}^{2m}$ if and only if $C$ has constant width. In this case, $\D_{C}$ is injective among $\mathcal{W}^{2m}$ if and only if $c_{2j+1}(h_{C})\neq 0$ for every $j\in\N\backslash\{0\}$.
\end{corollary}
\begin{proof}The convex body $\D_{C}K$ has constant width if and only if $\pi_{2j}(h_{\D_{C}K})=0$ for every $j>0$ (see Example \ref{examples}.iv). Using the orthogonal decomposition  \eqref{orthHk} of $\hH^{2m}_{2j}$, Funk-Hecke theorem and Equation \eqref{multipliers}, we get
$$\pi_{2j}(h_{\D_{C}\!K})=\sum_{l=0}^{2j}\pi_{l,2j-l}(h_{\D_{C}K})=\sum_{l=0}^{2j}\lambda_{l,2j-l}(A_{C})\pi_{l,2j-l}(h_{K})=\sum_{l=0}^{2j}c_{2(l-j)}(S_{C})\pi_{l,2j-l}(h_{K}),$$ 
and the results follow. 
\end{proof}

For the classical case, we have that the image of the class of bodies with constant width is the class of balls. This fact also follows directly from \ref{examples}.vi, and it remains true for any symmetric convex body $C$ with $c_{2j}(h_{C})\neq 0$, $j\in\N$. 
For a general $C$, $\image\D_{C}(\mathcal{W}^{2m})\subseteq\mathcal{W}^{2m}$. 

\smallskip
Next we devote ourselves to the more restrictive case of \textbf{balls}. First we study for which $C\in \K(\C)$ and $K\in\K(\C^m)$, the 
complex difference body is a ball, obtaining the following result, which is a consequence of Example \ref{examples}.i) and Corollary \ref{corCoeff}, and gives the reciprocal of Corollary \ref{cor:symCtnt}.

\begin{corollary}\label{c: dck ball}
Let $C\in\K(\C)$ and $K\in\K(\C^m)$. Then, $\D_{C}K$ is a ball if and only if  
for every $(k,l)\in\N^2 \backslash\{(0,0)\}$, either $c_{k-l}(h_{C})=0$ or $\pi_{k,l}(h_{K})=0$.

In particular, $D_{B_2}$ is injective among the class of $S^1$-invariant convex bodies. 
\end{corollary}

\begin{remark}If we consider $m=1$ in the previous corollary, i.e., $\D_{C}:\K(\C)\to\K(\C)$, it can be restated as follows: 
\begin{center}
\emph{$\D_{C}K$ is a ball if and only if for every $j\in\N \backslash\{0\}$, either $c_{j}(h_{C})=0$ or $c_{j}(h_{K})=0$.}
\end{center}
The previous condition coincides with the one found by G\"ortler in \cite{goertler} studying the relation between two (planar) convex bodies $K_{1}$ and $K_{2}$ in order
\begin{equation}\label{invRot}V_{2}(K_{1},K_{2})=V_{2}(K_{1},\rho K_{2})\quad\forall \rho\in\SO(2),
\end{equation}
to hold, i.e., the mixed volume to be invariant with respect to rotations on just one of the variables of the mixed volume 
(see also \cite[Theorem 4.6.2]{groemer.book}).

G\"ortler showed that \eqref{invRot} holds if and only if for every $j\geq 2$, either $c_{j}(h_{K_{1}})=0$ or $c_{j}(h_{K_{2}})=0$. Hence, we have that $\D_{C}K$ is a ball if and only if $V_{2}(C,K)=V_{2}(C,\rho K)$ for every $\rho\in\SO(2)$. Using the expression of the support function $h_{\D_{C}\!K}$ given in Proposition \ref{p: h eq vol mixto}, in particular \eqref{V2Rotation}, we clearly have the equivalence between both conditions. 
\end{remark}

\smallskip
Concerning the class of \textbf{$S^1$-invariant convex bodies}, we get from Lemma \ref{coeffS1} the following result, which, in particular, studies the reciprocal of the trivial fact that if $K$ is $S^1$-invariant, then $\dck$ is so too.

\begin{corollary}\label{cor:rtinv}Let $C\in\K(\C)$ and $K\in\K(\C^m)$. Then, $\D_{C}K$ is a $S^1$-invariant convex body  if and only if  for every $(k,l)\in\N^2$ with $k\neq l$, either $c_{k-l}(h_{C})=0$ or $\pi_{k,l}(h_{K})=0$.

In particular, $\image\D_{C}\subseteq\{\textrm{class of }S^1\textrm{-invariant convex bodies}\}$ if and only if $C$ is a ball. In this case, we have indeed equality.

Moreover, the operator $\D_{C}$ is injective among the class of $S^1$-invariant convex bodies for every $C$.
\end{corollary}

\smallskip
Analog to the study of convex bodies of constant width, we have, for the class of \textbf{symmetric convex bodies}:
\begin{corollary}\label{cor:osim}Let $C\in\K(\C)$ and $K\in\K(\C^m)$. Then, $\D_{C}K$ is symmetric if and only if for every $(k,l)\in\N^2$ with $k+l$ odd, either $c_{k-l}(h_{C})=0$ or $\pi_{k,l}(h_{K})=0$.

In particular, $\image\D_{C}\subseteq\{\textrm{symmetric convex bodies}\}$ if and only if $C$ is symmetric. In this case, the operator $\D_{C}$ is injective among all the symmetric convex bodies $K\in\K(\C^m)$ if and only if $c_{2j}(h_{C})\neq 0$, $j\in\N$.
\end{corollary}

\def\cprime{$'$}

\end{document}